\documentclass[DIV=14,letterpaper]{scrartcl}

\usepackage{tikz}
\usetikzlibrary{calc}
\usetikzlibrary{matrix}
\usepackage{amsmath,amssymb,amsfonts,amsthm}
\usepackage{graphicx}
\allowdisplaybreaks
\usepackage{subfigure,multicol,multirow}
\usepackage{makecell}
\usepackage{hyperref} 
\usepackage{diagbox}

\usepackage{bbm}

\theoremstyle{plain}

\newtheorem{pro}{\hspace{6mm}Proposition}[section]
\newtheorem{lem}{\hspace{6mm}Lemma}[section]

\theoremstyle{definition}

\theoremstyle{remark}

\newcommand{\V}[1]{\mathbf{#1}}

\newcommand{\email}[1]{\href{mailto:#1}{#1}}

\title{
A general regularization strategy for singular Stokes problems and convergence analysis for corresponding discretization and iterative solution 
}
\author{
Weizhang Huang\thanks{Department of Mathematics, the University of Kansas, 1460 Jayhawk Blvd, Lawrence, KS 66045, USA (\email{whuang@ku.edu}).}
\and
Zhuoran Wang\thanks{Department of Mathematics, the University of Kansas, 1460 Jayhawk Blvd, Lawrence, KS 66045, USA (\email{wangzr@ku.edu}).}
}

\date{} 

\begin{document}

\maketitle

\textbf{Abstract.}
A general regularization strategy is considered for the efficient iterative solution
of the lowest-order weak Galerkin approximation of singular Stokes problems.
The strategy adds a rank-one regularization term to the zero (2,2) block of the underlying
singular saddle point system. This strategy includes the existing pressure pinning and
mean-zero enforcement regularization as special examples. It is shown that
the numerical error maintains the optimal-order convergence provided that the nonzero
Dirichlet boundary datum is approximated numerically with sufficient accuracy.
Inexact block diagonal and triangular Schur complement preconditioners
are considered for the regularized system.
The convergence analysis for MINRES and GMRES with corresponding block preconditioners
is provided for different choices of the regularization term. Numerical experiments
in two and three dimensions are presented to verify the theoretical findings
and the effectiveness of the preconditioning for solving the regularized system. 

\vspace{10pt}

\noindent
\textbf{Keywords:}
Stokes flow, Regularization, MINRES and GMRES, Preconditioning, Weak Galerkin.

\vspace{5pt}

\noindent
\textbf{Mathematics Subject Classification (2020):}
65N30, 65F08, 65F10, 76D07

\section{Introduction}
\label{SEC:intro}
We consider a general regularization strategy for singular Stokes flow problems and their iterative solution
with block Schur complement preconditioning.
Let $ \Omega \subset \mathbb{R}^d$ $(d=2,3) $ be a bounded, connected polygonal/polyhedral domain with Lipschitz boundary $\partial \Omega$. The governing equation for the Stokes flow is given by:
\begin{equation}
\begin{cases}
\displaystyle
    -\mu \Delta \mathbf{u} + \nabla p  =  \mathbf{f},
      \quad \mbox{in } \; \Omega,
    \\
    \displaystyle
    \nabla \cdot \mathbf{u}  =  0,
     \quad \text{ in } \Omega,
    \\
    \displaystyle
    \mathbf{u}  =  \mathbf{g},
    \quad \mbox{on } \; \partial \Omega,
\end{cases}
\label{Eqn_StokesBVP}
\end{equation}
where
$ \mu>0 $ is the constant fluid kinematic viscosity,
$ \mathbf{u} $ is the fluid velocity,
$ p $ is the fluid pressure,
$ \mathbf{f} $ is a given source term,
$\mathbf{g}$ is a Dirichlet boundary datum of the velocity satisfying the compatibility condition
$ \int_{\partial \Omega}\mathbf{g}\cdot\mathbf{n}d S=0 $,
and $\mathbf{n}$ is the unit outward normal on
the boundary $\partial \Omega$.
Notice that (\ref{Eqn_StokesBVP}) is singular and the pressure solution is not unique.

Discretization of Stokes problems leads to algebraic saddle point systems. The iterative solution of those large-scale
systems has continued to gain attention from researchers;
e.g., see \cite{BenziGolubLiesen-2005,Benzi2008,Boffi-2008,Elman-2014}
and references therein. A major type of iterative solution methods is to employ Krylov subspace methods,
such as the minimal residual method (MINRES) and the generalized minimal residual method (GMRES),
with block Schur complement preconditioning. A key for this type of methods is to design effective preconditioners,
e.g., see \cite{BenziGolubLiesen-2005,Benzi2008,MurphyGolubWathen_SISC_2000,SilvesterWathen_SINUM_1994,WathenSilvester_SINUM_1993}
and recent works \cite{Ainsworth_SINUM_2022,Bacuta-2019,DassiScacchi_CMAME_2020,MeierBanschFrank_CMAME_2022,Rhebergen_SISC_2022}.
In particular, Silvester and Wathen \cite{SilvesterWathen_SINUM_1994,WathenSilvester_SINUM_1993} proposed to replace
the Schur complement with the pressure mass matrix in block Schur complement preconditioners and showed that
these inexact preconditioners are very effective for solving nonsingular Stokes flow problems.

For singular Stokes problems (including (\ref{Eqn_StokesBVP})), it is known (e.g., see \cite{Elman-2014,Vorst_2003})
that Krylov subspace methods work well when the systems are consistent. In this case, singular preconditioners can be used
\cite[Section 8.3.4]{Elman-2014}. Recently, Huang and Wang \cite{HuangWang_arxiv_2024_b} pointed out that the algebraic system
resulting from the discretization of (\ref{Eqn_StokesBVP}) is generally singular and inconsistent
for nonzero Dirichlet boundary datum $\V{g}$. They suggested to modify the underlying scheme to make the system consistent and
showed that nonsingular block Schur complement preconditioners with the (2,2) block taken as the pressure mass matrix
are effective for MINRES and GMRES solution of the singular system.
Regularization, which makes the underlying system nonsingular, is also often used for solving singular Stokes problems.
For example, Gwynllyw and Phillips \cite{GWYNLLYW20061027} used a regularization term to enforce the mean-zero condition
for the pressure variable. 
Another common approach is to prescribe a value of the pressure at a location (pinning).
A convergence analysis of MINRES and GMRES with block Schur complement preconditioning
was recently given in \cite{HuangWang_CiCP_2015} for a pinning regularization. 
Lee et al. \cite{LeeWuXuZika_2007} considered subspace correction methods for nearly singular systems.

In this work, we consider a general regularization strategy for the singular Stokes problem (\ref{Eqn_StokesBVP})
and study the effect of the regularization on the accuracy of the numerical solution and how efficiently the regularized system
can be solved using MINRES and GMRES with block Schur complement preconditioning. More specifically, we consider
the lowest-order weak Galerkin (WG) finite element method for the discretization of (\ref{Eqn_StokesBVP}).
We choose the WG method here because it satisfies the inf-sup or LBB condition (for stability) without using stabilization terms and
has the optimal-order convergence. Moreover, its error in the velocity is independent of the error in the pressure
(pressure-robustness) and the error in the pressure is independent of the viscosity $\mu$ ($\mu$-semi-robustness)
(see \cite{WANG202290}). It is worth emphasizing that while we use the WG method here, the regularization strategy and analysis
presented in this work can also be applied to other discretization methods provided that they satisfy the inf-sup condition.

The general regularization strategy we consider in this work adds a rank-one regularization term, $-(\rho/\mu) \V{w} \V{w}^T$,
to the zero (2,2) block in the original singular system resulting from the WG discretization
of (\ref{Eqn_StokesBVP}). Here, $\rho > 0$ is a positive parameter and $\V{w}$ is an arbitrarily chosen
unit vector not orthogonal to the null space of the gradient operator (cf. (\ref{w_def})).
The regularization can be viewed as an attempt to project the pressure solution 
into a subspace different from the null space of the gradient operator. Moreover, it can be
implemented efficiently through vector-vector multiplication.

We shall show that the regularized system is nonsingular and that its numerical solution
maintains the optimal convergence order provided that the nonzero Dirichlet boundary
datum is numerically approximated with sufficient accuracy.
Moreover, we consider block diagonal and triangular Schur complement preconditioners.
The convergence of MINRES and GMRES with these preconditioners depends on $\gamma = \V{w}^T\V{1}$ and thus the choice of $\V{w}$, where $\V{1}$ (cf. (\ref{vector-1}))
is a basis vector of the null space of the gradient operator.
We shall show that when $\gamma$ is finite, the convergence of MINRES and GMRES
with the corresponding block Schur complement preconditioners is independent
of $\mu$ and $h$, where $h$ is the maximum element diameter of the underlying
quasi-uniform mesh. For small $\gamma$, the convergence factor of preconditioned MINRES
and GMRES is independent of $\mu$ and $h$
while the asymptotic error constant is inverse proportional to $\gamma^2 \rho h^{-d}$
(and thus the number of iterations required for convergence depends
logarithmically on $\gamma^2 \rho h^{-d}$).

It is instructional to point out that the pressure mean-zero enforcement \cite{GWYNLLYW20061027} and pinning \cite{HuangWang_CiCP_2015}
are special examples (with finite and small $\gamma$, respectively)
of the general regularization strategy considered in this work.
In \cite{GWYNLLYW20061027}, Gwynllyw and Phillips showed that the pressure
mean-zero enforcement does not affect the optimal convergence order of
the underlying spectral element approximations. They numerically studied
the condition number of the Uzawa operator and the preconditioned
conjugate gradient solution of the corresponding systems but did not provide
an analytical analysis of the iterative solution.
Huang and Wang \cite{HuangWang_CiCP_2015} considered the pinning regularization
and presented a convergence analysis of the preconditioned MINRES and GMRES for
the corresponding regularized system.

The rest of this paper is organized as follows.
In Section~\ref{SEC:formulation}, we introduce the WG discretization for Stokes flow
and discuss the singularity and consistency of the resulting linear system.
In Section~\ref{sec::unify}, we consider a general regularization strategy to make the system nonsingular and show that the regularized system has optimal-order convergence and how
the solution error is affected by the regularization.
In Section~\ref{sec:finite-gamma}, we study the inexact block diagonal and block triangular Schur complement preconditioning, as well as the convergence of MINRES and GMRES for the regularized system with finite $\gamma$ and $\rho$. The situation with small $\gamma$
and $\rho$ is considered in Section~\ref{sec:small-gamma}.
Section~\ref{SEC:numerical} presents numerical experiments in two and three dimensions, along with comparisons of different regularization choices, to verify the theoretical findings and demonstrate the effectiveness of the proposed preconditioning.
Conclusions are drawn in Section~\ref{SEC:conclusions}.


\section{Weak Galerkin discretization}
\label{SEC:formulation}

In this section we give a brief description of the lowest-order WG discretization for
Stokes flow problem (\ref{Eqn_StokesBVP}). 
We also discuss the singularity and consistency of the resulting linear algebraic system.

\subsection{WG discretization for (\ref{Eqn_StokesBVP})}

The weak formulation of (\ref{Eqn_StokesBVP}) is to find 
$ \mathbf{u} \in H^1(\Omega)^d$ and $ p \in L^2(\Omega) $
such that $ \mathbf{u}|_{\partial\Omega} = \mathbf{g} $ (in the weak sense) and
\begin{equation}
\begin{cases}
  \mu (\nabla\mathbf{u}, \nabla\mathbf{v}) - (p, \nabla\cdot\mathbf{v})
= (\mathbf{f}, \mathbf{v}),
   \quad \forall \mathbf{v} \in H^1_0(\Omega)^d,
  \\ 
  -(\nabla\cdot\mathbf{u}, q)
  =  0,
 \quad \forall q \in L^2(\Omega) .
\end{cases}
\label{VarForm}
\end{equation}
Assume that a quasi-uniform simplicial mesh $\mathcal{T}_h = \{K\}$  is given for $\Omega$. Denote the maximum element diameter by $h$.
Define the discrete weak function spaces as
\begin{align}
     \displaystyle
    \mathbf{V}_h
    & = \{ \mathbf{u}_h = \{ \mathbf{u}_h^\circ, \mathbf{u}_h^\partial \}: \;
      \mathbf{u}_h^\circ|_{K} \in P_0(K)^d, \;
      \mathbf{u}_h^\partial|_e \in P_0(e)^d, \;
      \forall K \in \mathcal{T}_h, \; e \in \partial K \},
    \\ 
    \displaystyle
    W_h &= \{ p_h \in L^2(\Omega): \; p_h|_{K} \in P_0(K), \; \forall K \in \mathcal{T}_h\},
\end{align}
where $P_0(K)$ and $P_0(e)$ denote the spaces of constant polynomials defined on element $K$ and facet $e$, respectively.
Note that $\mathbf{u}_h \in \mathbf{V}_h$ is approximated on both interiors and facets of mesh elements, 
whereas $p_h \in W_h$ is approximated only on element interiors.
It is useful to point out that the mass matrix on $W_h$ is diagonal, i.e.,
\begin{align}
    \label{mass-1}
    M_p^{\circ} = \text{diag}(|K_1|, ..., |K_N|),
\end{align}
where $K_j$, $j = 1, ..., N$ denote the elements of $\mathcal{T}_h$ and $|K_j|$ denotes the volume of $K_j$.

For a scalar function or a component of a vector-valued function, $u_h = (u_h^{\circ},u_h^{\partial})$,
the discrete weak gradient operator $\nabla_w: W_h \rightarrow RT_0(\mathcal{T}_h)$ is defined as
\begin{equation}
\label{weak-grad-1}
  (\nabla_w u_h, \mathbf{w})_K
  = (u^\partial_h, \mathbf{w} \cdot \mathbf{n})_{\partial K}
  - ( u^\circ_h , \nabla \cdot \mathbf{w})_K,
  \quad \forall \mathbf{w} \in RT_0(K),\quad \forall K \in \mathcal{T}_h ,
\end{equation}
where $RT_0(K)$ is the lowest-order Raviart-Thomas space, i.e.,
\[RT_0(K) = (P_0(K))^d + \V{x} \, P_0(K),\]
$\mathbf{n}$ is the unit outward normal to $\partial K$ and $(\cdot, \cdot)_K$
and $( \cdot, \cdot )_{\partial K}$ are the $L^2$ inner product on $K$ and $\partial K$, respectively.
The analytical expression of $\nabla_w u_h$ can be found by choosing $\V{w}$ in (\ref{weak-grad-1}) appropriately; e.g., see \cite{HuangWang_CiCP_2015}.
For a vector-valued function $\mathbf{u}_h$, $\nabla_w \mathbf{u}_h$ is treated as a matrix, where each row corresponds to the weak gradient
of a component.
The discrete weak divergence operator $\nabla_w \cdot: \mathbf{V}_h \to \mathcal{P}_0(\mathcal{T}_h)$ is defined
separately as 
\begin{equation}
   (\nabla_w \cdot \mathbf{u}, w )_{K}
  = ( \mathbf{u}^\partial , w \mathbf{n})_{ e }
  - ( \mathbf{u}^\circ , \nabla w)_{K},
  \quad
  \forall w \in P_0(K) .
  \label{wk_div1}
\end{equation}
Note that $\nabla_w \cdot \V{u}|_K \in P_0(K)$. Taking $w = 1$, we have
\begin{equation}
(\nabla_w \cdot \V{u}, 1)_{K} = \sum_{i=1}^{d+1} |e_{K,i}| \langle\V{u}\rangle_{e_{K,i}}^T \V{n}_{K,i} ,
\label{wk_div2}
\end{equation}
where $\langle\V{u}\rangle_{e_{K,i}}$ denotes the average of $\V{u}$ on facet $e_{K,i}$
and $|e_{K,i}|$ is the $(d-1)$-dimensional measure of $e_{K,i}$.

Having defined the discrete weak spaces, gradient, and divergence, we can introduce the WG approximation
of (\ref{VarForm}) as: finding $ \mathbf{u}_h \in \mathbf{V}_h $ and $ p_h \in W_h $
such that $ \mathbf{u}_h^\partial|_{\partial \Omega} = Q_h^{\partial}\mathbf{g} $ and
\begin{equation}
\begin{cases}
    \displaystyle
    \mu \sum_{K\in\mathcal{T}_h} (\nabla_w \mathbf{u}_h,\nabla_w \mathbf{v})_K 
    -\sum_{K \in \mathcal{T}_h}(p_h^{\circ}, \nabla_w\cdot\mathbf{v})_K
    =  \sum_{K \in \mathcal{T}_h} (\mathbf{f}, \mathbf{\Lambda}_h\mathbf{v})_K,
      \quad \forall \mathbf{v} \in \mathbf{V}_h^0,
    \\ 
    \displaystyle
    -\sum_{K \in \mathcal{T}_h}(\nabla_w\cdot\mathbf{u}_h,q^{\circ})_K
    =  0,
   \quad \forall q \in W_h,
\end{cases}
\label{scheme}
\end{equation}
where $Q_h^{\partial}$ is a $L^2$-projection operator onto $\V{V}_h$ restricted on each facet
and $\mathbf{\Lambda}_h: \V{V}_h \to RT_0(\mathcal{T}_h)$ is the lifting operator \cite{Mu.2020,WANG202290} defined as
\begin{equation}
    \displaystyle
    ( (\mathbf{\Lambda}_h\mathbf{v}) \cdot \mathbf{n}, w )_{e}
      = ( \mathbf{v}^{\partial}\cdot\mathbf{n}, w )_{e},
        \quad \forall w \in P_0(e),\; \forall \V{v} \in \V{V}_h, \; \forall e \subset \partial K .
\label{Eqn_DefLambdah}
\end{equation}
Notice that $\mathbf{\Lambda}_h\mathbf{v}$ depends on $\mathbf{v}^{\partial}$
but not on $\mathbf{v}^{\circ}$. 



Denote the WG approximation of $\V{u}$ on the interior and facets of any element $K$
by $\V{u}_{h,K}^{\circ}$ and $\V{u}_{h,K,i}^{\partial} \; (i = 1,...,d+1)$, respectively. We can express $\V{u}_h$ as
\begin{align*}
\displaystyle
\V{u}_{h}(\V{x}) & = \V{u}_{h,K}^{\circ} \varphi_K^{\circ}(\V{x}) + \sum_{i = 1}^{d+1} \V{u}_{h,K,i}^{\partial}
\varphi_{K,i}^{\partial}(\V{x})
\\
&
= \V{u}_{h,K}^{\circ} \varphi_K^{\circ}(\V{x})
 + \sum_{\substack{i = 1\\ e_{K,i} \notin \partial \Omega}}^{d+1} \V{u}_{h,K,i}^{\partial}
\varphi_{K,i}^{\partial}(\V{x})
+ \sum_{\substack{i = 1\\ e_{K,i} \in \partial \Omega}}^{d+1} \V{u}_{h,K,i}^{\partial}
\varphi_{K,i}^{\partial}(\V{x}),
\quad \forall \mathbf{x} \in K ,\; \forall K \in \mathcal{T}_h .
\end{align*}
Then, we can cast \eqref{scheme} into a matrix-vector form as
\begin{equation}
    \begin{bmatrix}
        \mu A & -(B^{\circ})^T \\
       -B^{\circ} & \mathbf{0}
    \end{bmatrix}
    \begin{bmatrix}
        \mathbf{u}_h \\
        \mathbf{p}_h
    \end{bmatrix}
    =
    \begin{bmatrix}
        \mathbf{b}_1 \\
        \mathbf{b}_2
    \end{bmatrix},
    \label{scheme_matrix}
\end{equation}
where the matrices $A$ and $B^{\circ}$ and vectors $\V{b}_1$ and $\V{b}_2$ are defined as
\begin{align}
    \mathbf{v}^T A \mathbf{u}_h  & =  \sum_{K\in\mathcal{T}_h} (\nabla_w \mathbf{u}_h,\nabla_w \mathbf{v})_K
    \label{A-1}
    \\
    & \displaystyle = \sum_{K\in\mathcal{T}_h} (\V{u}_{h,K}^{\circ}\nabla_w \varphi_K^{\circ},  \nabla_w \mathbf{v})_K
     + \sum_{K\in\mathcal{T}_h} \sum^{d+1}_{\substack{i = 1 \\e_{K,i} \notin \partial \Omega}}(\V{u}_{h,K,i}^{\partial}\nabla_w \varphi_{K,i}^{\partial},\nabla_w \mathbf{v})_K,
    \quad \forall \mathbf{u}_h, \mathbf{v} \in \mathbf{V}_h^0,
    \notag \\
\mathbf{q}^T B^{\circ} \mathbf{u}_h  & =  \sum_{K \in \mathcal{T}_h}  (\nabla_{w}\cdot\mathbf{u}_h,q^{\circ})_K
    = \sum_{K\in\mathcal{T}_h} \sum^{d+1}_{\substack{i = 1 \\e_{K,i} \notin \partial \Omega}}
    |e_{K,i}| q_{K}^{\circ}(\V{u}_{h,K,i}^{\partial})^T \V{n}_{K,i},
    \quad \forall \mathbf{u}_h \in \mathbf{V}_h^0, \quad \forall q \in W_h,
    \label{B-1} 
    \\
\mathbf{v}^T \V{b}_1 & = \sum_{K \in \mathcal{T}_h} (\mathbf{f}, \mathbf{\Lambda}_h\mathbf{v})_K
    - \mu \sum_{K\in\mathcal{T}_h} \sum_{\substack{i = 1\\ e_{K,i} \in \partial \Omega}}^{d+1}( (Q_{h}^{\partial}\V{g}) \nabla_w \varphi_{K,i}^{\partial},\nabla_w \mathbf{v})_K, \; \forall \mathbf{v} \in \mathbf{V}_h^0,
    \label{b1-1}
    \\
\mathbf{q}^T \V{b}_2 & = \sum_{K\in\mathcal{T}_h} \sum^{d+1}_{\substack{i = 1 \\e_{K,i} \in \partial \Omega}}
    |e_{K,i}| q_{K}^{\circ} (Q_{h}^{\partial}\V{g}|_{e_{K,i}} )^T \V{n}_{K,i},
    \quad \forall q \in W_h .
    \label{b2-1}
\end{align}
Here, $\V{v}$ (or $\V{v}_h$) denotes both the WG approximation of $\V{v}_h\in \V{V}_h$
and its vector representation, consisting of $(\V{v}_{h,K}^{\circ}, \V{v}_{h,K,i}^{\partial})$ for $i = 1, ..., d+1$ and $K \in \mathcal{T}_h$,
excluding those on $\partial \Omega$.
Similarly, for any $q_h \in W_h$, the vector formed by $q_{h,K}$ for all $K \in \mathcal{T}_h$ is represented by $\V{q}$ (or $\V{q}_h$).
Note that $ (B^{\circ})^T$ is the WG approximation to the gradient operator.

It is known (see, e.g., \cite{{WANG202290}}) that the scheme (\ref{scheme})
is pressure robust and has optimal-order convergence provided that the compatibility condition is satisfied (see the discussion in Section~\ref{sec:consistency}).
Moreover, it satisfies the inf-sup condition as stated in the following lemma.
\begin{lem}\label{inf_sup}
There exists a constant $ 0<\beta<1 $ independent of the parameter $\mu$ and the mesh $\mathcal{T}_h$ such that 
\begin{align}
   \sup_{\mathbf{v}\in \mathbf{V}_h^0, \quad \mathbf{v}^T A \mathbf{v} \neq 0}
  \frac{\V{v}^T(B^{\circ})\V{p}}
      {(\mathbf{v}^T A \mathbf{v})^{\frac{1}{2}}} 
    \geq \beta \|\V{p}\|, 
  \qquad \forall p \in W_h.
\end{align}
\end{lem}
\begin{proof}
The proof can be found in \cite[Lemma 7]{WANG202290}.
\end{proof}


\subsection{Singularity and consistency of the linear algebraic system}
\label{sec:consistency}


\begin{lem}
\label{lem:B0-1}
The null space of $(B^{\circ})^T$ is given by
\[
\text{Null}((B^{\circ})^T) = \{ p_h \in W_h: \; p_{h,K} = C, \; \forall K \in \mathcal{T}_h, \; \text{ C is a constant} \} .
\]
\end{lem}

\begin{proof}
The proof can be found in \cite[Lemma 2.1]{HuangWang_CiCP_2025}.
\end{proof}

The above lemma implies that $(B^{\circ})^T$ is deficient by one rank
and the linear system
(\ref{scheme_matrix}) is singular.

The system (\ref{scheme_matrix}) is also inconsistent for general nonzero $\V{g}$. To show this, we define
\begin{align}
    \label{vector-1}
\V{1}= \frac{1}{\sqrt{N}}\begin{bmatrix} 1 \\ \vdots \\ 1 \end{bmatrix},
\end{align}
where $N$ is the number of the elements of the mesh $\mathcal{T}_h$.
Lemma~\ref{lem:B0-1} implies that $\V{1}$ is a basis vector of the null space of $(B^{\circ})^T$.
Multiplying the second block equation of (\ref{scheme_matrix}) with $\V{1}^T$ (from left) and using (\ref{b2-1}), we get
\[
0 = \frac{\alpha_h}{\sqrt{N}} ,
\]
where
\begin{align}
    \alpha_h = \sqrt{N}\, \V{1}^T \V{b}_2 =  \sum_{e \in \partial \Omega} \int_{e} (Q_h^\partial \V{g}|_e)^T \V{n}_e d S 
    = \sum_{e \in \partial \Omega} |e|\Big( \V{Q}_h^{\partial}\V{g} |_e - \langle \V{g} \rangle_e \Big)^T {\V{n}_e},
\label{alpha-1}
\end{align}
where we have used the compatibility condition $\int_{\partial \Omega} \V{g}^T \V{n} d S = 0$.
Since $\V{Q}_h^{\partial}\V{g} |_e \neq \langle \V{g} \rangle_e $ generally for nonzero $\V{g}$,
$\alpha_h$ is not zero. As a consequence, (\ref{scheme_matrix}) is inconsistent in general.

On the other hand, $\V{Q}_h^{\partial}\V{g} |_e$ can be taken as the value of $\V{g}$ at the center of facet $e$
or computed by applying a Gaussian quadrature rule to $\langle \V{g} \rangle_e$. Thus, 
the approximation $\V{Q}_h^{\partial}\V{g} |_e$ to $\langle \V{g} \rangle_e $ can be made
to be second-order or higher, i.e., 
\begin{align}
\alpha_h = \mathcal{O}(h^q), \quad \text{with} \quad q \ge 2.
\label{alpha-3}
\end{align}

Singularity and inconsistency pose challenges in solving the linear system (\ref{scheme_matrix}).
Remedies for this include making the scheme consistent (while still being singular) or making it nonsingular.
For example, Huang and Wang \cite{HuangWang_arxiv_2024_b} suggested to modify $\V{b}_2$ to make the scheme consistent.
They showed that the modification does not affect the optimal convergence order of the numerical solution. They also showed that
the modified numerical scheme (which is singular but consistent) can be solved efficiently by MINRES and GMRES with block Schur
complement preconditioning.
Indeed, it is known (e.g., see \cite[Remark 6.12]{Elman-2014} and \cite[Section 10.2]{Vorst_2003})
that Krylov subspace methods, such as MINRES and GMRES, work well for consistent singular systems.

The other approach is using regularization, i.e., making the numerical scheme nonsingular.
For example, Gwynllyw and Phillips \cite{GWYNLLYW20061027} used a regularization term to enforce the mean-zero condition for the pressure variable.
Moreover, it is common, especially when direct solvers are used, to specify a value of the pressure at a suitable location
(this technique is often referred as pinning); see, e.g., \cite{Bochev_SIAM_2005,HuangWang_CiCP_2025}.

In this work we study a general strategy for regularization, which includes the pinning and mean-zero enforcement techniques as special cases.

\section{A general regularization strategy}
\label{sec::unify}

In this section we present a general strategy for the regularization of scheme (\ref{scheme_matrix}).
Since the regularization can affect the accuracy of the numerical solution, we also provide a convergence analysis
for the regularized scheme in this section.

Assume that a unit vector $\V{w} \in \mathbb{R}^N$ has been chosen to satisfy
\begin{align}
    \V{w}^T \V{w} = 1, \quad \V{w}^T \V{1} \neq 0,
    \label{w_def}
\end{align}
where $\V{1}$ is defined in \eqref{vector-1}. For notational convenience, denote $\gamma = \V{w}^T \V{1} \neq 0$.
The general regularization strategy is to add a regularization term, $-\frac{\rho}{\mu} \V{w}\V{w}^T$,
to the zero (2,2)-block of system \eqref{scheme_matrix}, leading to the new system
\begin{align}
     \begin{bmatrix}
        \mu A & -(B^{\circ})^T \\
       -B^{\circ} &  -\frac{\rho}{\mu} \V{w}\V{w}^T
    \end{bmatrix}
    \begin{bmatrix}
        {\mathbf{u}}_h \\
        {\mathbf{p}}_h
    \end{bmatrix}
    =
    \begin{bmatrix}
        \mathbf{b}_1 \\
        \mathbf{b}_2
    \end{bmatrix},
    \label{scheme_reg}
\end{align}
where $\rho$ is a positive constant and its choice will be discussed later.
Notice that the solution of this regularized system is different from that of the original system \eqref{scheme_matrix}
in general. Without confusion and for notational simplicity, we use the same notation for the solutions of both systems.
Moreover, although $\V{w} \V{w}^T$ is a full matrix, its multiplication with vectors and therefore, the multiplication of the coefficient matrix of (\ref{scheme_reg}) with vectors, can be implemented efficiently.

For the regularized system (\ref{scheme_reg}), we would like to know how
the accuracy of the numerical solution is affected by the regularization
and how efficiently the system can be solved iteratively using MINRES and GMRES with block Schur complement preconditioning.
For the former, we give a convergence analysis for (\ref{scheme_reg}) in this section
while the latter issue is studied in Sections~\ref{sec:finite-gamma} and \ref{sec:small-gamma}.

For the convergence analysis of (\ref{scheme_reg}), multiplying its second block equation with $\V{1}^T$ (from left) 
and noticing $\V{1}^T B^{\circ} = \V{0}^T$, we get
\begin{align}
    - \frac{\rho}{\mu} \V{w}^T \V{p}_h = \frac{\alpha_h}{\gamma \sqrt{N}} ,
    \label{projection-0}
\end{align}
where $\alpha_h$ is defined in (\ref{alpha-1}) and $\gamma = \V{w}^T\V{1} $.
From this and (\ref{alpha-3}), we can reasonably interpret
the regularization in (\ref{scheme_reg}) as an attempt to enforce $\V{w}^T \V{p}_h = 0$.
In this sense, we can regard the regularization as a projection technique that projects $\V{p}_h$
into the orthogonal complement of $\text{span}(\V{w})$.
Since $\V{w}$ is not orthogonal to $\V{1}$, the orthogonal complement of $\text{span}(\V{w})$
is different from $\text{span}(\V{1})$ (i.e., the null space of $(B^{\circ})^T$).

Using (\ref{projection-0}), we can rewrite (\ref{scheme_reg}) into
\begin{align}
     \begin{bmatrix}
        \mu A & -(B^{\circ})^T \\
       -B^{\circ} &  \V{0}
    \end{bmatrix}
    \begin{bmatrix}
        {\mathbf{u}}_h \\
        {\mathbf{p}}_h
    \end{bmatrix}
    =
    \begin{bmatrix}
        \mathbf{b}_1 \\
        \mathbf{b}_2 - \frac{\alpha_h}{\gamma \sqrt{N} } \V{w} 
    \end{bmatrix} .
    \label{scheme_reg-2}
\end{align}
This system is similar to the modified numerical scheme studied in \cite{HuangWang_arxiv_2024_b} to ensure the consistency of the linear system.
The following proposition can be established using a proof similar to that for \cite[Theorem 3.1]{HuangWang_arxiv_2024_b}.

\begin{pro}
\label{thm:reg_scheme_err}
Let $ \mathbf{u} \in H^{2}(\Omega)^d $ and $ p \in H^1(\Omega) $ be the solutions of Stokes problem (\ref{VarForm}) and 
$\mathbf{u}_h\in \mathbf{V}_h$ and $p_h\in W_h$ be the numerical solutions of the regularized scheme (\ref{scheme_reg}).
Assume $ \mathbf{f} \in L^2(\Omega)^d $.
Then,
\begin{align}
    & \| p - p_h \| \le C h \| \V{f} \| + C \mu \frac{|\alpha_h|}{|\gamma|},
    \label{thm:reg_scheme_err-1}
    \\
   &  \| \nabla \V{u} - \nabla_w \V{u}_h \| \le C h \|\V{u} \|_2 +  C \frac{|\alpha_h|}{|\gamma|},
   \label{thm:reg_scheme_err-3}
    \\
    & \|\mathbf{u} - \mathbf{u}_h\| = \|\mathbf{u} - \mathbf{u}_h^\circ\| \leq C h \|\mathbf{u}\|_{2}
    + C \frac{|\alpha_h|}{|\gamma|} ,
    \label{thm:reg_scheme_err-2}
    \\
      & 
   \displaystyle \|Q^{\circ}_h \V{u} - \V{u}^{\circ}_h\|  
  \le Ch^2 \| \V{u}\|_2 + C \frac{|\alpha_h|}{|\gamma|},
   \label{thm:reg_scheme_err-2-0}
\end{align}
where $C$ is a constant independent of $h$ and $\mu$.
\end{pro}

The above proposition shows how the error in the numerical solution is affected by the regularization through $\alpha$ and $\gamma$.
Recall from (\ref{alpha-1}) that $\alpha_h$ measures how closely the Dirichlet boundary datum $\V{g}$ is approximated by $Q_h^\partial \V{g}|_e$.
Thus, the effect of the regularization on the accuracy of the numerical solution is proportional to the error in approximating
the Dirichlet boundary datum.
Particularly, when $\gamma$ is a finite number, achieving the optimal convergence order of the numerical solution
requires the approximation to be $\mathcal{O}(h^2)$ at least.
A higher-order approximation to the boundary datum is needed when $\gamma$ is small.

\section{Convergence analysis of MINRES and GMRES with block Schur complement preconditioning with finite $\gamma$ and $\rho$}
\label{sec:finite-gamma}

Our next task is to study how efficiently the regularized scheme (\ref{scheme_reg}) can be solved by MINRES and GMRES
with block Schur complement preconditioning.
In this section we consider the situation where $\gamma$ and $\rho$ are finite
in the sense that they are constants (independent of $\mu$ and $h$ particularly).
In this case, the convergence of MINRES and GMRES with block Schur complement preconditioning can be shown to be
essentially independent of $\mu$ and $h$ (see Propositions~\ref{pro:MINRES_conv} and \ref{pro:GMRES_conv} below).

Two special examples of finite $\gamma$ have been studied by other researchers.
Gwynllyw and Phillips \cite{GWYNLLYW20061027} studied the enforcement of the mean-zero pressure condition in the 
spectral element approximation of the Stokes problem. In the current notation, this method corresponds to the choice
\begin{align}
    \label{mean-zero-1}
    \V{w} = \frac{\sqrt{N}}{\sqrt{\sum\limits_{K\in \mathcal{T}_h} |K|^2}}M_p^{\circ} \V{1}, \quad 
    \gamma = \frac{|\Omega|}{\sqrt{N} \sqrt{\sum\limits_{K\in \mathcal{T}_h} |K|^2}} = \mathcal{O}(1).
\end{align}

The other example is
\begin{align}
    \label{projection-1}
    \V{w} = \V{1}, \quad  \gamma = 1.
\end{align}
Recall that $\V{1}$ is a basis vector of the null space of $(B^{\circ})^T$. Thus, this choice attempts to project $\V{p}_h$
into the orthogonal complement of the null space of $(B^{\circ})^T$. A similar regularization was used by
Bochev and Lehoucq \cite{Bochev_SIAM_2005} for pure Neumann problems of Poisson's equation where 
$\V{1}$ is a basis vector of the null space of the coefficient matrix.

To start the analysis for general $\V{w}$, we rescale the unknown variables and rewrite \eqref{scheme_reg} as
\begin{equation}
    \begin{bmatrix}
        A & -(B^{\circ})^T \\
       -B^{\circ} &  - \rho \V{w} \V{w}^T
    \end{bmatrix}
    \begin{bmatrix}
        \mu \mathbf{u}_h \\
        \mathbf{p}_h
    \end{bmatrix}
    =
    \begin{bmatrix}
        \mathbf{b}_1 \\
        \mu \mathbf{b}_2
    \end{bmatrix},
    \quad \mathcal{A}   = \begin{bmatrix}
        A & -(B^{\circ})^T \\
       -B^{\circ} &  - \rho \V{w} \V{w}^T
    \end{bmatrix} .
    \label{scheme_matrix_2}
\end{equation}
The Schur complement for this system is $ S =  \rho\V{w} \V{w}^T +  B^{\circ} A^{-1} (B^{\circ})^T$.

\subsection{Schur complement and its eigenvalue estimation}
\label{sec:eigenvalue}

\begin{lem}
\label{lem:S-22}
The Schur complement $S$
is symmetric and positive semi-definite. Moreover, it satisfies
\begin{align}
    0 \le S \le d \;( \rho \V{w} \V{w}^T + M_p^{\circ}),
    \label{lem:S-bound-12}
\end{align}
where the sign ``$\le$'' between matrices is in the sense of negative semi-definite. 
\end{lem}
\begin{proof}
It is obvious that $S$ is symmetric and positive semi-definite and thus $0 \le S$.
From the definitions of $B^{\circ}$ \eqref{B-1} and $M_p^{\circ}$ (\ref{mass-1})
and the fact $\nabla_w \cdot \V{u} \in P_0(\mathcal{T}_h)$, we have
\begin{align*}
\V{u}^{T}  (B^{\circ})^T (M_p^{\circ})^{-1} B^{\circ} \V{u} = \sum_{K \in \mathcal{T}_h}(\nabla_w \cdot \V{u}, \nabla_w \cdot \V{u})_K .
\end{align*}
Moreover, it can be verified directly that
\begin{align*}
   \sum_{K \in \mathcal{T}_h}(\nabla_w \cdot \mathbf{u_h},\nabla_w \cdot \mathbf{u_h})_K \le
d \sum_{K \in \mathcal{T}_h}(\nabla_w \mathbf{u}_h,\nabla_w \mathbf{u}_h )_K .
\end{align*}
Then, since both $A$ (cf. \eqref{scheme_matrix}) and $M_p^{\circ}$ are symmetric and positive definite, we have
\begin{align*}
    \sup_{\mathbf{p} \neq 0} \frac{\V{p}^{T}( B^{\circ}{A}^{-1}(B^{\circ})^T) \V{p}}{\V{p}^T M_p^{\circ} \V{p}} 
    & = \sup_{\mathbf{p} \neq 0}  \frac{\V{p}^{T} (M_p^{\circ})^{-\frac12} ( B^{\circ}{A}^{-1}(B^{\circ})^T)
    (M_p^{\circ})^{-\frac12} \V{p}}{\V{p}^T \V{p}}
    \notag
    \\
    & \le \sup_{\mathbf{u} \neq 0} \frac{\V{u}^{T}  (B^{\circ})^T (M_p^{\circ})^{-1} B^{\circ}
    \V{u}}{\V{u}^T A \V{u}}
    \notag \\
    & = \sup_{\mathbf{u} \neq 0} \frac{\sum_{K \in \mathcal{T}_h}(\nabla_w \cdot \mathbf{u_h},\nabla_w \cdot \mathbf{u_h})_K}{\sum_{K \in \mathcal{T}_h}(\nabla_w \mathbf{u}_h,\nabla_w \mathbf{u}_h )_K} \le  d,
\end{align*}
or
\begin{align}
    \sup_{\mathbf{p} \neq 0} \frac{\V{p}^{T}( B^{\circ}{A}^{-1}(B^{\circ})^T) \V{p}}{\V{p}^T M_p^{\circ} \V{p}} 
    \le  d. 
    \label{AB-1}
\end{align}
Thus, we have
\begin{align*}
    \V{p}^T S \V{p} & =  \V{p}^T \rho \V{w} \V{w}^T \V{p} + \V{p}^T  B^{\circ}{A}^{-1}(B^{\circ})^T \V{p}
    \leq \V{p}^T d \;( \rho \V{w} \V{w}^T  + M_p^{\circ}) \V{p},
\end{align*}
which leads to the right inequality of (\ref{lem:S-bound-12}).
\end{proof}

Motivated by Lemma~\ref{lem:S-22}, we take an approximation of the Schur complement $S$ as
\begin{align}
\label{hatS-1}
    \hat{S} =  \rho\V{w} \V{w}^T  + M_p^{\circ} .
\end{align}
Using the Sherman-Morrison-Woodbury formula, we can obtain the expression of the inverse of $\hat{S}$ as
\begin{align}
    \hat{S}^{-1} =  (M_p^{\circ})^{-1} - \frac{\rho (M_p^{\circ})^{-1} \V{w} \V{w}^T (M_p^{\circ})^{-1}}{1+ \rho \V{w}^T (M_p^{\circ})^{-1} \V{w}} .
    \label{hatS_inv}
\end{align}
Notice that $M_p^{\circ}$ is diagonal so its inverse is trivial to compute.
Moreover, the multiplication of $\hat{S}^{-1}$ with vectors can be implemented efficiently.

Now we establish bounds for the eigenvalues of $\hat{S}^{-1} S$. Those bounds will be needed in the convergence analysis
for MINRES and GMRES.

\begin{lem}
\label{eigen_bound2}
The eigenvalues of $\hat{S}^{-1} S$ are bounded by
\begin{align}
\label{eigen_bound2-eq}
C_1 + \mathcal{O}(h^{d}) \le \lambda_i(\hat{S}^{-1} S) \le d, \quad i = 1, ..., N 
\end{align}
where
\begin{align}
    \label{C1-1}
    C_1 = \displaystyle  \beta^2 \; \frac{\lambda_{\min}(M_p^{\circ})}{\lambda_{\max}(M_p^{\circ})} \; \gamma^2 .
\end{align}
\end{lem}
\begin{proof}
The proof is given in Appendix~\ref{sec:eigen_bound2-proof}.
\end{proof}

Since the mesh is assumed to be quasi-uniform, we have ${\lambda_{\min}(M_p^{\circ})}/{\lambda_{\max}(M_p^{\circ})} = \mathcal{O}(1)$
and $C_1 = \mathcal{O}(1)$. 
Then, Lemma~\ref{eigen_bound2} implies that the eigenvalues of $\hat{S}^{-1} S$
are bounded below and above essentially by positive constants.
Since $\hat{S}$ is symmetric and positive definite,
it follows that $S$, and therefore the coefficient matrix $\mathcal{A}$ of the regularized system (\ref{scheme_matrix_2}), are
non-singular.


\subsection{Convergence of MINRES with block diagonal Schur complement preconditioning}
\label{sec:mires}

Now, we study the iterative solution of the regularized system \eqref{scheme_matrix_2} with MINRES and the block diagonal Schur complement preconditioner
\begin{align}
    \mathcal{P}_d = 
    \begin{bmatrix}
        A & 0 \\
        0 & \hat{S}
    \end{bmatrix} , \quad \hat{S} =  \rho\V{w} \V{w}^T  + M_p^{\circ}.
    \label{PrecondP-diag}
\end{align}
Since $\mathcal{P}_d$ is symmetric and positive definite,
the preconditioned coefficient matrix $\mathcal{P}_d^{-1} \mathcal{A}$ is similar to the symmetric matrix
$\mathcal{P}_d^{-1/2} \mathcal{A} \mathcal{P}_d^{-1/2}$, which can be expressed as
\begin{align}
    \mathcal{P}_d^{-\frac{1}{2}} \mathcal{A} \mathcal{P}_d^{-\frac{1}{2}}&= 
    \begin{bmatrix}
        A^{\frac{1}{2}} & 0 \\
        0 &\hat{S}^{\frac{1}{2}}
    \end{bmatrix}^{-1}
    \begin{bmatrix}
        A & -(B^{\circ})^T \\
       -B^{\circ} & -\rho \V{w}\V{w}^T
       \end{bmatrix}
    \begin{bmatrix}
        A^{\frac{1}{2}} & 0 \\
        0 &\hat{S}^{\frac{1}{2}}
    \end{bmatrix}^{-1} \notag
    \\ 
    & = \begin{bmatrix}
        \mathcal{I} & -A^{-\frac{1}{2}} (B^{\circ})^T\hat{S}^{-\frac{1}{2}} \\
        -\hat{S}^{-\frac{1}{2}} B^{\circ}A^{-\frac{1}{2}} &  -\rho \hat{S}^{-\frac{1}{2}} \V{w}\V{w}^T  \hat{S}^{-\frac{1}{2}}
    \end{bmatrix}.
    \label{diag-precond-system}
\end{align}

\begin{lem}
\label{lem:eigen_bound_diag}
The eigenvalues of $ \mathcal{P}_d^{-1} \mathcal{A} $ lie in 
\begin{align}
  &   
      \Bigg[ -\frac{d}{\sqrt{C_1}  } + \mathcal{O}(h^d), -\frac{2 C_1 }{1+\sqrt{1+4d}} + \mathcal{O}(h^d) \Bigg] 
      \bigcup
\Bigg[\sqrt{C_1} + \mathcal{O}(h^d), 
\;  \frac{1}{2}(1+\sqrt{1+4d})\Bigg],
 \label{lem:eigen_bound_diag-1}
\end{align}
where $C_1$ is defined in (\ref{C1-1}).
\end{lem}

\begin{proof}
    The eigenvalue problem of the preconditioned system $ \mathcal{P}_d^{-1} \mathcal{A} $ reads as
   \begin{align}
    \begin{bmatrix}
        A & -(B^{\circ})^T \\
       -B^{\circ} &  -\rho \V{w}\V{w}^T
    \end{bmatrix}
    \begin{bmatrix}
       \V{u} \\
        \V{p}
    \end{bmatrix} = \lambda
   \begin{bmatrix}
        A & 0 \\
        0 & \hat{S}
    \end{bmatrix} 
        \begin{bmatrix}
      \V{u} \\
        \V{p}
    \end{bmatrix} .
    \label{mu0system}
\end{align}
It is not difficult to show that $\lambda = 1$ is not an eigenvalue.
Solving the first equation for $\V{u}$ and substituting it into the second equation, we get
\begin{align*}
    \lambda^2 \hat{S} \V{p} - \lambda M^{\circ}_p \V{p} - S \V{p} = \V{0}.
\end{align*}
From this, we obtain 
\begin{align*}
    \lambda^2  - \lambda \frac{\V{p}^TM^{\circ}_p \V{p}}{\V{p}^T \hat{S} \V{p}} - \frac{\V{p}^T S \V{p}}{\V{p}^T \hat{S} \V{p}} = 0,
\end{align*}
which leads to
\[
\lambda_{\pm} = \frac{1}{2} \frac{\V{p}^T M^{\circ}_p\V{p}}{\V{p}^T \hat{S} \V{p}} \pm \frac{1}{2}\sqrt{ (\frac{\V{p}^T M^{\circ}_p \V{p}}{\V{p}^T \hat{S} \V{p}})^2
+ 4 \frac{\V{p}^T S\V{p}}{\V{p}^T \hat{S} \V{p}} }.
\]
Lemma~\ref{eigen_bound2} implies
\begin{align}
    C_1 + \mathcal{O}(h^{d})
    \le \frac{\V{p}^T S\V{p}}{\V{p}^T \hat{S} \V{p}} \le d .
    \label{eigen_bound_diag3}
\end{align}
Moreover, we have
\begin{align*}
    \frac{\V{p}^TM_p^{\circ}\V{p}}{\V{p}^T \hat{S} \V{p}} = \frac{\V{p}^T M_p^{\circ}\V{p}}{\V{p}^T (\rho \V{w}\V{w}^T + M_p^{\circ}) \V{p}} \leq 1
\end{align*}
and
\begin{align*}
     \frac{\V{p}^T \hat{S} \V{p}}{\V{p}^TM_p^{\circ}\V{p}} \leq 
     1 + \rho\frac{\V{p}^T \V{w}\V{w}^T \V{p} }{\V{p}^T M_p^{\circ} \V{p}}
     \leq 1 + \frac{\rho}{\lambda_{\min} (M_p^{\circ})},
\end{align*}
which lead to
\begin{align}
 \frac{\lambda_{\min} (M_p^{\circ})}{\rho+ \lambda_{\min} (M_p^{\circ})}  \leq 
 \frac{\V{p}^TM_p^{\circ}\V{p}}{\V{p}^T \hat{S} \V{p}} \leq 1.
 \label{eigen_bound_diag0}
\end{align}
Using \eqref{eigen_bound_diag3} and \eqref{eigen_bound_diag0}, we know that the positive eigenvalues are bounded by
\[
\Big[\sqrt{C_1} + \mathcal{O}(h^d), \;  \frac{1}{2}(1+\sqrt{1+4d})\Big],
\]
where we have used $\frac{\lambda_{\min} (M_p^{\circ})}{\rho+ \lambda_{\min} (M_p^{\circ})} = \mathcal{O}(h^d)$.

The negative eigenvalues can be rewritten as
\begin{align*}
 \lambda_{-} & =
\frac{ -2 \frac{\V{p}^T S\V{p}}{\V{p}^T \hat{S} \V{p}}}{\frac{\V{p}^TM_p^{\circ}\V{p}}{\V{p}^T \hat{S} \V{p}}  
 + 
 \sqrt{(\frac{\V{p}^TM_p^{\circ}\V{p}}{\V{p}^T \hat{S} \V{p}})^2 + 4  \frac{\V{p}^T S\V{p}}{\V{p}^T \hat{S} \V{p}} } }   .
\end{align*}
From \eqref{eigen_bound_diag3} and \eqref{eigen_bound_diag0}, $\lambda_{-}$ is bounded by
\begin{align*}
    \Big[ -\frac{d}{\sqrt{C_1}  } + \mathcal{O}(h^d), -\frac{2 C_1 }{1+\sqrt{1+4d}} + \mathcal{O}(h^d) \Big].
\end{align*}
These provide the bounds for the eigenvalues of $ \mathcal{P}_d^{-1} \mathcal{A} $.
\end{proof}




\begin{pro}
\label{pro:MINRES_conv}
The residual of MINRES applied to the preconditioned system $\mathcal{P}_{d}^{-1} \mathcal{A} $ is bounded by   
\begin{align}
    \frac{\| \V{r}_{2k}\| }{\| \V{r}_0 \|} 
    \le 2 \left( \frac{\sqrt{d} (1+\sqrt{1+4d})-2 C_1}{\sqrt{d} (1+\sqrt{1+4d})+2 C_1} + \mathcal{O}(h^d)\right)^k ,
     \label{pro:MINRES_conv-1}
\end{align}
where $C_1$ is given in (\ref{C1-1}).
\end{pro}

\begin{proof}
As shown in Lemma~\ref{lem:eigen_bound_diag}, the eigenvalues of the preconditioned system $\mathcal{P}_d^{-1}\mathcal{A}$ are bounded by
$[-a_1,-b_1] \cup [c_1,d_1]$,
where
\begin{align*}
  & a_1 =  \frac{d}{\sqrt{C_1}  } + \mathcal{O}(h^d)  ,
  \qquad
   b_1 = \frac{2 C_1 }{1+\sqrt{1+4d}} + \mathcal{O}(h^d)  ,
   \\
   & c_1 = \sqrt{C_1} + \mathcal{O}(h^d),  
    \qquad
   d_1 = \frac{1}{2}(1+\sqrt{1+4d}).
\end{align*}
From \cite[Theorem 6.13]{Elman-2014}, after $2k$ steps of MINRES iteration, the residual satisfies
\begin{align*}
    \| \V{r}_{2k}\| 
& \leq 2 \left( \frac{\sqrt{\frac{a_1d_1}{b_1c_1}} - 1}{\sqrt{\frac{a_1d_1}{b_1c_1}} + 1}\right)^k \| \V{r}_0 \|
= 2 \left( \frac{\sqrt{d}(1+\sqrt{1+4d})-2 C_1}{\sqrt{d} (1+\sqrt{1+4d})+2 C_1} + \mathcal{O}(h^d)\right)^k \| \V{r}_0 \|.
\end{align*}
This gives \eqref{pro:MINRES_conv-1}.
\end{proof}

Proposition~\ref{pro:MINRES_conv} shows that on a quasi-uniform mesh, where $ C_1 = \mathcal{O}(1)$,
the convergence of MINRES for the preconditioned system $ \mathcal{P}_d^{-1} \mathcal{A} $ is essentially
independent of the parameters $h $ and $\mu$.


\subsection{Convergence of GMRES with block triangular Schur complement preconditioning}
\label{sec:gmres}

Now we study the convergence of GMRES for the regularized system (\ref{scheme_matrix_2})
with the block triangular Schur complement preconditioner
\begin{align}
    \mathcal{P}_t = 
    \begin{bmatrix}
        A & 0 \\
        -B^{\circ} & -\hat{S}
    \end{bmatrix},\quad \hat{S} =  \rho\V{w} \V{w}^T  + M_p^{\circ}.
    \label{PrecondP-tri}
\end{align}
We present the analysis only for the lower preconditioner \eqref{PrecondP-tri}, as the block upper and lower triangular preconditioners perform similarly.

\begin{pro}
    \label{pro:GMRES_conv}
The residual of GMRES applied to the preconditioned system $\mathcal{P}_{t}^{-1} \mathcal{A}$ is bounded by
\begin{align}
\frac{\| \V{r}_k\|}{\| \V{r}_0\|} 
\le 
2\left(1+\left (\frac{d\, \lambda_{\max} (M_p^{\circ})}{\lambda_{\min} (A)}\right )^\frac{1}{2} + d\right) 
\left(\frac{\sqrt{d} - \displaystyle 
 \sqrt{C_1}}{\sqrt{d} + \displaystyle \sqrt{C_1 } } \; + \mathcal{O}(h^d)\right)^{k-1} ,
\label{GMRES-residual-5}
\end{align}
where $C_1$ is given in (\ref{C1-1}).
\end{pro}

\begin{proof}
    From \cite[Lemma~A.1]{HuangWang_CiCP_2025}, 
    the residual of GMRES for the preconditioned system $\mathcal{P}_{t}^{-1} \mathcal{A}$ is bounded as 
    \begin{align}
        \frac{\| \V{r}_k \|}{\|\V{r}_0\|} \le
        (1+\|A^{-1}(B^{\circ})^T\| + \| \hat{S}^{-1} S \|) \min\limits_{\substack{p \in \mathbb{P}_{k-1}\\ p(0) = 1}} \| p(\hat{S}^{-1} S) \| ,
        \label{bound0}
        \end{align}
where $\mathbb{P}_{k-1}$ denotes the set of polynomials of degree up to $k-1$.
Lemma~\ref{eigen_bound2} implies
\[
\| \hat{S}^{-1} S \| \le d .
\]
Moreover, since $A$ and $M_p^{\circ}$ are symmetric and positive definite,  we have
\begin{align*}
    \| A^{-1} (B^{\circ})^T \|^2 
    &= \sup_{\V{p} \neq 0} \frac{\V{p}^T B^{\circ} A^{-1} A^{-1} (B^{\circ})^T \V{p}}{\V{p}^T  \V{p}} \nonumber
    \\
   & = \sup_{\V{p} \neq 0} \frac{\V{p}^T (M_p^{\circ})^{\frac{1}{2}} (M_p^{\circ})^{-\frac{1}{2}}B^{\circ} A^{-1}A^{-1} (B^{\circ})^T (M_p^{\circ})^{-\frac{1}{2}} (M_p^{\circ})^{\frac{1}{2}} \V{p}}{\V{p}^T \V{p}}
   \notag \\
   & \le \lambda_{\max} (A^{-1}) \lambda_{\max} (M_p^{\circ})
 \sup_{\V{u} \neq 0} \frac{\V{u}^T B^{\circ} (M_p^{\circ})^{-1} (B^{\circ})^T \V{u}}{\V{u}^T A\V{u}}
 \notag \\
 & \le \frac{d\, \lambda_{\max} (M_p^{\circ})}{\lambda_{\min} (A)} .
\end{align*}

For the minmax problem in \eqref{bound0}, by shifted Chebyshev polynomials (e.g., see \cite[Pages 50-52]{Greenbaum-1997})
and Lemma~\ref{eigen_bound2}, we have
\begin{align*}
   \min\limits_{\substack{p \in \mathbb{P}_{k-1}\\ p(0) = 1}} \| p(\hat{S}^{-1}S) \|
& = \min\limits_{\substack{p \in \mathbb{P}_{k-1}\\ p(0) = 1}} \max_{i=1,..., N} |p(\lambda_i(\hat{S}^{-1}S))|
\le \min\limits_{\substack{p \in \mathbb{P}_{k-1}\\ p(0) = 1}} \max_{\gamma \in \left[ C_1 + \mathcal{O}(h^d),d\right]} |p(\gamma)| 
\\
& \leq 2 \left(\frac{\sqrt{d} - \displaystyle 
 \sqrt{C_1}}{\sqrt{d} + \displaystyle \sqrt{C_1 } } \; + \mathcal{O}(h^d)\right)^{k-1} .
\end{align*}
Combining the above results we obtain \eqref{GMRES-residual-5}.
\end{proof}

Recall that ${\lambda_{\max} (M_p^{\circ})}/{\lambda_{\min} (A)} = \mathcal{O}(h^{d-2})$
for a quasi-uniform mesh (e.g., see \cite{LinLiuFarrah_JCAM_2015}). Thus, Proposition~\ref{pro:GMRES_conv} shows that
the convergence of GMRES with the block triangular preconditioner (\ref{PrecondP-tri})
for the iterative solution of the regularized system (\ref{scheme_matrix_2}) is essentially independent of $\mu$ and $h$.

\section{Convergence analysis of MINRES and GMRES with small $\gamma$ and $\rho$}
\label{sec:small-gamma}

In this section we consider the situation where $\gamma$ and $\rho$ are small.
In this situation, we can consider the regularization term as a perturbation
to the original singular system. More specifically, for both the preconditioned Schur complement and coefficient matrix of the whole system, we estimate the non-small eigenvalues using the Bauer-Fike theorem (e.g., see \cite[Corollary 6.5.8]{Watkins-2010})
and the small eigenvalue using its derivative with respect to $\rho$. 
For notational simplicity, for the current situation we choose the approximation of the Schur complement
$S =  \rho\V{w} \V{w}^T +  B^{\circ} A^{-1} (B^{\circ})^T$ as
\begin{align}
    \label{hatS-2}
    \hat{S} = M_p^{\circ} .
\end{align}

It is worth pointing out that a special example of small $\gamma$ is the pinning technique with which
a value of the pressure is specified at a location. 
When the pinning location is chosen to the centroid of the first element (e.g., see \cite{HuangWang_CiCP_2025}),
this corresponds to the choice of $\V{w}$
\begin{align}
    \label{w-pinnning}
    \V{w} = [1, 0, ..., 0]^T .
\end{align}
For this case, we have $\gamma = 1/\sqrt{N}$.

\subsection{Eigenvalue estimation for the preconditioned Schur complement}

\begin{lem}
\label{lem:eigen_bound_2}
Assume that $\rho$ is small and satisfies
\begin{align}
   \rho < \beta^2 \lambda_{\min} (M_p^{\circ}) .
    \label{lem:eigen_bound_2-1}
\end{align}
Then, the first eigenvalue of $\hat{S}^{-1} S$ is given by
\begin{align}
\lambda_1(\hat{S}^{-1} S) = \frac{\gamma^2 N}{|\Omega|} \rho + \mathcal{O}(\rho^2),
    \label{lem:eigen_bound_2-3}
\end{align}
and the other eigenvalues are bounded by
\begin{align}
    \beta^2 - \frac{\rho}{\lambda_{\min}(M_p^{\circ})} \le \lambda_i(\hat{S}^{-1} S) \le d + \frac{\rho}{\lambda_{\min}(M_p^{\circ})},
    \quad i = 2, ..., N.
    \label{lem:eigen_bound_2-4}
\end{align}
\end{lem}

\begin{proof}
Notice that $\hat{S}^{-1} S$ is similar to 
\begin{align}
    \hat{S}^{-\frac{1}{2}} \; S \; \hat{S}^{-\frac{1}{2}}
= (M_p^{\circ})^{-\frac{1}{2}} B^{\circ} A^{-1}(B^{\circ})^T (M_p^{\circ})^{-\frac{1}{2}}
+ \rho (M_p^{\circ})^{-\frac{1}{2}} \V{w}\V{w}^T (M_p^{\circ})^{-\frac{1}{2}} .
\label{lem:eigen_bound_2-5}
\end{align}
The term involving $\rho$ on the right-hand side can be viewed as a perturbation to the first term, which has
a zero eigenvalue and the other eigenvalues lying in $[\beta^2, d]$, where $\beta$ is the inf-sup constant;
e.g., see \cite{HuangWang_CiCP_2025}.
It is not difficult to show 
\[
\| \rho (M_p^{\circ})^{-\frac{1}{2}} \V{w}\V{w}^T (M_p^{\circ})^{-\frac{1}{2}} \| \leq \frac{\rho}{\lambda_{\min} (M_p^{\circ})}.
\]
The Bauer-Fike theorem (e.g., see \cite[Corollary 6.5.8]{Watkins-2010}) implies that the nonzero eigenvalues 
of \eqref{lem:eigen_bound_2-5} are bounded by (\ref{lem:eigen_bound_2-4}).
Notice that the lower bound is positive when (\ref{lem:eigen_bound_2-1}) is satisfied.

For the smallest eigenvalue $\lambda_1(\rho) = \lambda_1(\hat{S}^{-1} S)$, we consider the eigenvalue problem
\begin{align*}
    \Big( (M_p^{\circ})^{-\frac 1 2} B^{\circ}(A)^{-1} (B^{\circ})^T (M_p^{\circ})^{-\frac 1 2} + \rho (M_p^{\circ})^{-\frac{1}{2}} \V{w}\V{w}^T (M_p^{\circ})^{-\frac{1}{2}}  \Big) \V{p} (\rho) = \lambda_1(\rho) \V{p}(\rho),
\end{align*}
which is equivalent to
\begin{align}
    \Big(  B^{\circ}(A)^{-1} (B^{\circ})^T + \rho \V{w}\V{w}^T \Big) \tilde{\V{p}} (\rho) = \lambda_1(\rho) M_p^{\circ} \tilde{\V{p}}(\rho),
    \label{lem:eigen_bound_2-6}
\end{align}
with 
$ \tilde{\V{p}}(\rho) =(M_p^{\circ})^{-\frac 1 2}  \V{p}(\rho)$.
Notice that when $\rho = 0$, we have $\lambda_1 (0) = 0$ and $\tilde{\V{p}}(0) = \V{1}$.
Differentiating \eqref{lem:eigen_bound_2-6} with respect to $\rho$, we have
\begin{align}
     \Big(  B^{\circ}(A)^{-1} (B^{\circ})^T + \rho \V{w}\V{w}^T \Big) \frac{\partial \tilde{\V{p}}}{\partial \rho} (\rho)  + \V{w}\V{w}^T \; \tilde{\V{p}} (\rho)  = \frac{\partial \lambda_1}{\partial \rho} (\rho) M_p^{\circ}\; \tilde{\V{p}}(\rho) + 
     \lambda_1(\rho)  M_p^{\circ}\; \frac{\partial \tilde{\V{p}}}{\partial \rho} (\rho) .
    \label{lem:eigen_bound_2-7}
\end{align}
Taking $\rho = 0$ in \eqref{lem:eigen_bound_2-7}, multiplying from the left by $\V{1}^T$, and noticing that  $\V{1}$ is in the null space of $(B^{\circ})^T$
and $\gamma = \V{w}^T \V{1}$, we have
\begin{align*}
\frac{\partial \lambda_1}{\partial \rho} (0) = 
\frac{\gamma^2}{ \V{1}^T  M_p^{\circ} \V{1}} .
\end{align*}
This, together with $\V{1}^T  M_p^{\circ} \V{1}  = \frac{|\Omega|}{N}$, gives \eqref{lem:eigen_bound_2-3}.
\end{proof}


\subsection{Convergence of MINRES}

Now, we study the iterative solution of the regularized system \eqref{scheme_matrix_2} using MINRES and the block diagonal Schur complement preconditioner
\begin{align}
    \mathcal{P}_d = 
    \begin{bmatrix}
        A & 0 \\
        0 & \hat{S}
    \end{bmatrix} ,\quad \hat{S} = M_p^{\circ} .
    \label{PrecondP-diag-small}
\end{align}

\begin{lem}
    \label{lem:eigen_bound_diag2}
Assume that $\rho$ is small and satisfies 
\begin{align}
    \label{lem:eigen_bound_diag2-1}
    \frac{\rho}{\lambda_{\min} (M_p^{\circ})} \le \frac{\sqrt{1+4 \beta^2 }-1}{2} .
\end{align}
Then the eigenvalues of $\mathcal{P}_d^{-1} \mathcal{A}$ lie in 
\begin{align}
&\Big [\frac{1-\sqrt{1+4 d }}{2} - \frac{\rho}{\lambda_{\min} (M_p^{\circ})} , \frac{1-\sqrt{1+4 \beta^2}}{2} + \frac{\rho}{\lambda_{\min} (M_p^{\circ})}\Big ]
\bigcup \Big \{ -\frac{\gamma^2 N}{|\Omega|} \rho + \mathcal{O}(\rho^2) \Big \}
\notag \\
& \qquad \qquad \qquad \qquad \bigcup \Big [ \frac{1+\sqrt{1+4 \beta^2}}{2} - \frac{\rho}{\lambda_{\min} (M_p^{\circ})},
\frac{1+\sqrt{1+4 d}}{2} + \frac{\rho}{\lambda_{\min} (M_p^{\circ})} \Big ] .
 \label{lem:eigen_bound_diag2-0}
\end{align}

\end{lem}

\begin{proof}
Notice that $\mathcal{P}_d^{-1} \mathcal{A}$ is similar to
\begin{align*}
    \mathcal{P}_d^{-\frac{1}{2}} \mathcal{A} \mathcal{P}_d^{-\frac{1}{2}}&= 
    \begin{bmatrix}
        A^{\frac{1}{2}} & 0 \\
        0 &(M_p^{\circ})^{\frac{1}{2}}
    \end{bmatrix}^{-1}
    \begin{bmatrix}
        A & -(B^{\circ})^T \\
       -B^{\circ} & -\rho \V{w}\V{w}^T
       \end{bmatrix}
    \begin{bmatrix}
        A^{\frac{1}{2}} & 0 \\
        0 &(M_p^{\circ})^{\frac{1}{2}}
    \end{bmatrix}^{-1} \notag
    \\ 
    & = \begin{bmatrix}
        \mathcal{I} & -A^{-\frac{1}{2}} (B^{\circ})^T(M_p^{\circ})^{-\frac{1}{2}} \\
        -(M_p^{\circ})^{-\frac{1}{2}} B^{\circ}A^{-\frac{1}{2}} &  -\rho (M_p^{\circ})^{-\frac{1}{2}} \V{w}\V{w}^T  (M_p^{\circ})^{-\frac{1}{2}}
    \end{bmatrix}.
\end{align*}
The rest of the proof is similar to that of Lemma~\ref{lem:eigen_bound_2}. More specifically, the term involving $\rho$ is considered
as a perturbation to the rest. Moreover, the eigenvalue closest to the origin is obtained by differentiating the eigenvalue equation with respect
to $\rho$ and the other eigenvalues are estimated using Lemma~\ref{lem:eigen_bound_2} (for the eigenvalues of the unperturbed system) and with the Bauer-Fike theorem for the perturbed system.
\end{proof}

\begin{pro}
    \label{pro:MINRES_conv-small}
Assume that $\rho$ is small and satisfies \eqref{lem:eigen_bound_diag2-1}. Then, the residual of MINRES applied to the preconditioned system $\mathcal{P}_{d}^{-1} \mathcal{A} $ is bounded by
\begin{align}
    \frac{\| \V{r}_{2k+1}\| }{\| \V{r}_0 \|} 
    &\le 2 \cdot \frac{\frac{1+\sqrt{1+4d}}{2}+ \frac{\rho}{\lambda_{\min} (M_p^{\circ})} + 
   \frac{\gamma^2 N}{|\Omega|} \rho + \mathcal{O}(\rho^2)}{ \frac{\gamma^2 N}{|\Omega|} \rho + \mathcal{O}(\rho^2) }  \notag
  \\
 & \cdot
     \Bigg( \frac{\sqrt{(\frac{\rho}{\lambda_{\min} (M_p^{\circ})} )^2 + \frac{\rho}{\lambda_{\min} (M_p^{\circ})} \sqrt{1+4d} + d } - 
     \sqrt{(\frac{\rho}{\lambda_{\min} (M_p^{\circ})} )^2 - \frac{\rho}{\lambda_{\min} (M_p^{\circ})} \sqrt{1+4\beta^2} + \beta^2 } 
     }{\sqrt{(\frac{\rho}{\lambda_{\min} (M_p^{\circ})} )^2 + \frac{\rho}{\lambda_{\min} (M_p^{\circ})} \sqrt{1+4d} + d } +
     \sqrt{(\frac{\rho}{\lambda_{\min} (M_p^{\circ})} )^2 - \frac{\rho}{\lambda_{\min} (M_p^{\circ})} \sqrt{1+4\beta^2} + \beta^2 } } \Bigg)^k.
\label{pro:MINRES_conv-small-1}
\end{align}
\end{pro}

\begin{proof}
As shown in \cite{MINRES-1975}, the residual of MINRES is given by
\[
\| \V{r}_{2k+1}\| = \min\limits_{\substack{p \in \mathbb{P}_{2k+1}\\ p(0) = 1}} \| p (\mathcal{P}_{d}^{-1}\mathcal{A})  \V{r}_0 \|
\le \min\limits_{\substack{p \in \mathbb{P}_{2k+1}\\ p(0) = 1}} \| p (\mathcal{P}_{d}^{-1} \mathcal{A})\| \; \| \V{r}_0 \|,
\]
where $\mathbb{P}_{2k+1}$ is the set of polynomials of degree up to $2k+1$.
Moreover, bounds on the eigenvalues of $\mathcal{P}_{d}^{-1} \mathcal{A} $ are given in Lemma~\ref{lem:eigen_bound_diag2}.
We denote the eigenvalue intervals by $[-a_1,-b_1]\cup [c_1,d_1]$ and the eigenvalue close to zero by $\lambda_1$.
From Theorem 6.13 of \cite{Elman-2014} (about the residual of MINRES), we have
\begin{align}
    \frac{\| \V{r}_{2k+1}\| }{\| \V{r}_0 \|} 
    &\le \min\limits_{\substack{p \in \mathbb{P}_{2k+1}\\ p(0) = 1}} \max\limits_{ i = 1,2,..,2N-1} | p (\lambda_i)| \notag
\\
 &\le
    \min\limits_{\substack{p \in \mathbb{P}_{2k}\\ p(0) = 1}} 
    \max_{i = 2, ..., 2 N-1} |\frac{(\lambda_i - \lambda_1)}{\lambda_1} p(\lambda_i) | \nonumber
    \\
   &\le \frac{\frac{1+\sqrt{1+4d}}{2}+ \frac{\rho}{\lambda_{\min} (M_p^{\circ})} + 
   \frac{\gamma^2 N}{|\Omega|} \rho+ \mathcal{O}(\rho^2)}{ \frac{\gamma^2 N}{|\Omega|} \rho + \mathcal{O}(\rho^2)}
   \min\limits_{\substack{p \in \mathbb{P}_{2k}\\ p(0) = 1}} 
    \max_{\lambda \in  [-a_1,-b_1]\cup [c_1,d_1]} | p(\lambda) | \nonumber 
    \\
  &   \le 2 \frac{\frac{1+\sqrt{1+4d}}{2}+ \frac{\rho}{\lambda_{\min} (M_p^{\circ})} + 
   \frac{\gamma^2 N}{|\Omega|} \rho+ \mathcal{O}(\rho^2)}{ \frac{\gamma^2 N}{|\Omega|} \rho + \mathcal{O}(\rho^2)} \notag
  \\
  & \qquad \cdot
     \Bigg( \frac{\sqrt{(\frac{\rho}{\lambda_{\min} (M_p^{\circ})} )^2 + \frac{\rho}{\lambda_{\min} (M_p^{\circ})} \sqrt{1+4d} + d } - 
     \sqrt{(\frac{\rho}{\lambda_{\min} (M_p^{\circ})} )^2 - \frac{\rho}{\lambda_{\min} (M_p^{\circ})} \sqrt{1+4\beta^2} + \beta^2 } 
     }{\sqrt{(\frac{\rho}{\lambda_{\min} (M_p^{\circ})} )^2 + \frac{\rho}{\lambda_{\min} (M_p^{\circ})} \sqrt{1+4d} + d } +
     \sqrt{(\frac{\rho}{\lambda_{\min} (M_p^{\circ})} )^2 - \frac{\rho}{\lambda_{\min} (M_p^{\circ})} \sqrt{1+4\beta^2} + \beta^2 } } \Bigg)^k.
\notag
\end{align}
\end{proof}

Notice that (\ref{lem:eigen_bound_diag2-1}) implies $\rho = \mathcal{O}(\lambda_{\min}(M_p^{\circ})) = \mathcal{O}(h^d)$.
Then, the above proposition
shows that the convergence factor of MINRES for the case with small $\gamma$ and $\rho$,
\[
\frac{\sqrt{(\frac{\rho}{\lambda_{\min} (M_p^{\circ})} )^2 + \frac{\rho}{\lambda_{\min} (M_p^{\circ})} \sqrt{1+4d} + d } - 
     \sqrt{(\frac{\rho}{\lambda_{\min} (M_p^{\circ})} )^2 - \frac{\rho}{\lambda_{\min} (M_p^{\circ})} \sqrt{1+4\beta^2} + \beta^2 } 
     }{\sqrt{(\frac{\rho}{\lambda_{\min} (M_p^{\circ})} )^2 + \frac{\rho}{\lambda_{\min} (M_p^{\circ})} \sqrt{1+4d} + d } +
     \sqrt{(\frac{\rho}{\lambda_{\min} (M_p^{\circ})} )^2 - \frac{\rho}{\lambda_{\min} (M_p^{\circ})} \sqrt{1+4\beta^2} + \beta^2 } },
\]
is essentially constant and thus independent
of $\mu$ and $h$. On the other hand, the asymptotic error constant can be written as ${C}/{\gamma^2 N \rho}$,
where $C$ is a constant, which depends on $\gamma$, $\rho$, and $N$ (or $h$).
Since the number of MINRES iterations required to reach convergence is proportional to $\log (\gamma^2 N \rho)$, the dependence
on these parameters are weak.

For the pinning case (\ref{w-pinnning}), the asymptotic error constant can be simplified into $C/\rho$.

\subsection{Convergence of GMRES}

We consider the block lower triangular preconditioner
\begin{align}
    \mathcal{P}_t = 
    \begin{bmatrix}
        A & 0 \\
        -B^{\circ} & -\hat{S}
    \end{bmatrix},\quad \hat{S} =  M_p^{\circ}.
    \label{PrecondP-tri-small}
\end{align}

\begin{pro}
    \label{pro:GMRES_conv-small}
Assume that $\rho$ is small and satisfies \eqref{lem:eigen_bound_2-1}. Then, the residual of GMRES applied to the preconditioned system $\mathcal{P}_{t}^{-1} \mathcal{A} $ is bounded by 
\begin{align}
\frac{\| \V{r}_k\|}{\| \V{r}_0\|} 
 & \le
2 \Bigg (1+\left (\frac{d\, \lambda_{\max} (M_p^{\circ})}{\lambda_{\min} (A)}\right )^\frac{1}{2} + d \Bigg) 
\cdot
\left( \frac{d+\frac{\rho}{\lambda_{\min} (M_p^{\circ})} + \frac{\gamma^2 N}{|\Omega|}\rho + \mathcal{O}(\rho^2)}{\frac{\gamma^2 N}{|\Omega|}\rho + \mathcal{O}(\rho^2)}  \right)
\notag
\\
&
 \qquad \qquad \cdot \left (\frac{\sqrt{d + \frac{\rho}{\lambda_{\min} (M_p^{\circ})}} -\sqrt{\beta^2 - \frac{\rho}{\lambda_{\min} (M_p^{\circ})}}}{\sqrt{d + \frac{\rho}{\lambda_{\min} (M_p^{\circ})}} + \sqrt{\beta^2 - \frac{\rho}{\lambda_{\min} (M_p^{\circ})}} }\right )^{k-2} .
\label{pro:GMRES_conv-small-1}
\end{align}
\end{pro}

\begin{proof}
    We follow the proof of Proposition~\ref{pro:GMRES_conv} except that the outlier eigenvalue
    $\lambda_1$ in \eqref{lem:eigen_bound_2-3} is treated separately from the others that are bounded away from zero.
    Thus, we have
    \begin{align*}
    \min\limits_{\substack{p \in \mathbb{P}_{k-1}\\ p(0) = 1}} \| p(\hat{S}^{-1}S) \|
    & = 
     \min\limits_{\substack{p \in \mathbb{P}_{k-1}\\ p(0) = 1}} 
    \max_{i = 1,..., N} |p(\lambda_i)|
    \nonumber
    \\
    &\le
    \min\limits_{\substack{p \in \mathbb{P}_{k-2}\\ p(0) = 1}} 
    \max_{i = 2, ..., N} |\frac{(\lambda_i - \lambda_1)}{\lambda_1} p(\lambda_i) | \nonumber
    \\
    &\le \frac{d+\frac{\rho}{\lambda_{\min} (M_p^{\circ})}+\lambda_1}{\lambda_1}
    \min\limits_{\substack{p \in \mathbb{P}_{k-2}\\ p(0) = 1}} 
    \max_{ \lambda \in [\beta^2-\frac{\rho}{\lambda_{\min} (M_p^{\circ})},  
    d + \frac{\rho}{\lambda_{\min} (M_p^{\circ})} ]}| p(\lambda) | \nonumber
    \\
    & 
    \le 
    2
\left( \frac{d+\frac{\rho}{\lambda_{\min} (M_p^{\circ})}+ \frac{\gamma^2 N}{|\Omega|}\rho  + \mathcal{O}(\rho^2)}{\frac{\gamma^2 N}{|\Omega|}\rho  + \mathcal{O}(\rho^2)}  \right)
 \left (\frac{\sqrt{d + \frac{\rho}{\lambda_{\min} (M_p^{\circ})}} -\sqrt{\beta^2 - \frac{\rho}{\lambda_{\min} (M_p^{\circ})}}}{\sqrt{d + \frac{\rho}{\lambda_{\min} (M_p^{\circ})}} + \sqrt{\beta^2 - \frac{\rho}{\lambda_{\min} (M_p^{\circ})}} }\right )^{k-2} .
\end{align*}
The other terms are bounded as in Proposition~\ref{pro:GMRES_conv}.
\end{proof}

The condition \eqref{lem:eigen_bound_2-1} implies that $\rho = \mathcal{O}(\lambda_{\min}(M_p^{\circ})) = \mathcal{O}(h^d)$.
Proposition~\ref{pro:GMRES_conv-small} shows that the convergence factor of GMRES for small $\gamma$,
\[
\frac{\sqrt{d + \frac{\rho}{\lambda_{\min} (M_p^{\circ})}} -\sqrt{\beta^2 - \frac{\rho}{\lambda_{\min} (M_p^{\circ})}}}{\sqrt{d + \frac{\rho}{\lambda_{\min} (M_p^{\circ})}} + \sqrt{\beta^2 - \frac{\rho}{\lambda_{\min} (M_p^{\circ})}} },
\]
is essentially independent of $\mu$ and $h$. The asymptotic error constant can be written as $C/(\gamma^2 N \rho)$, which indicates
that the required number of GMRES iterations for convergence depends on $\gamma^2 N \rho$ but this dependence is only logarithmic.

For the pinning case (\ref{w-pinnning}), the asymptotic error constant can be simplified into $C/\rho$.

\section{Numerical experiments}
\label{SEC:numerical}

To demonstrate the performance of MINRES with the block diagonal preconditioner \eqref{PrecondP-diag} or (\ref{PrecondP-diag-small})
and GMRES with the block triangular Schur complement preconditioner \eqref{PrecondP-tri} or \eqref{PrecondP-tri-small},
we present two- and three-dimensional numerical results in this section.
We use MATLAB's function {\em minres} with $tol = 10^{-9}$ for 2D examples and  $tol = 10^{-8}$ for 3D examples with block diagonal preconditioners,  a maximum of 1000 iterations, and the zero vector as the initial guess.
For preconditioned systems with block triangular preconditioners, we use MATLAB's function {\em gmres} with $tol = 10^{-9}$ for 2D examples and  $tol = 10^{-8}$ for 3D examples, $restart = 30$, and the zero vector as the initial guess.
The implementation of block preconditioners requires the action of the inversion of blocks.
The inversion of $\hat{S}$ involves the inverse of the mass matrix $M_{p}^{\circ}$, which is diagonal and thus trivial to invert.
It also involves matrix-vector multiplications as shown in \eqref{hatS_inv} for the choice (\ref{hatS-1}).
The leading block $A$ is the WG approximation of the Laplacian operator, as shown in \eqref{A-1}. 
Its inversion is computed using the conjugate gradient method, preconditioned with an incomplete Cholesky decomposition.
The incomplete Cholesky decomposition is performed using MATLAB's function {\em ichol} with threshold
dropping and a drop tolerance of $10^{-3}$.

We test the performance of iterative solvers with four different choices of $\V{w}$,
where 
\[
\V{w}_1 = \V{1},  \quad
\V{w}_2 = \frac{\sqrt{N}}{(\sum_{i=1}^N   |K_i|^2)^{\frac{1}{2}}} M_p^{\circ} \V{1},  \quad
\V{w}_3 =  \begin{bmatrix}
        1 \\ 0 \\ \vdots \\ 0
    \end{bmatrix}, \quad 
\V{w}_4 = \mbox{Unit random vector in } [0,1].
\]
Among these, $\V{w}_1, \V{w}_2$, and $\V{w}_4$ correspond to finite values of $\gamma$ and $\rho$. 
In the computation, we set $\rho = 1$ for these cases.
The theoretical convergence results of MINRES and GMRES are presented
in Propositions~\ref{pro:MINRES_conv} and ~\ref{pro:GMRES_conv}, respectively.
On the other hand, $\V{w}_3$ represents a pinning case of small $\gamma$ and $\rho$. 
The convergence analysis of iterative solution is presented in Propositions~\ref{pro:MINRES_conv-small}
and \ref{pro:GMRES_conv-small} for MINRES and GMRES, respectively.
For this case, we take $\hat{S} = M_p^{\circ}$. Moreover, as required by the convergence analysis,
$\rho$ should be chosen to satisfy \eqref{lem:eigen_bound_diag2-1} for MINRES and
(\ref{lem:eigen_bound_2-1}) for GMRES. Numerically we found $\beta^2 \approx 0.2$. Thus, we take 
$\rho = 0.1 \times |K|_{\min}$ in our computation.


\subsection{The two-dimensional example}

This two-dimensional (2D) example is adopted from \cite{Mu.2020}, where
$\Omega = (0,1)^2$. 
The exact solutions are given by
\begin{align*}
\V{u} = 
\begin{bmatrix}
    -e^x(y \cos(y)+\sin(y)) \\
    e^x y \sin(y)
\end{bmatrix}, \quad
p = 2 e^x \sin(y),
\end{align*}
and the right-hand side function is
\begin{align*}
\V{f} = 
\begin{bmatrix}
    2(1-\mu) e^x  \sin(y) \\
    2(1-\mu) e^x \cos(y)
\end{bmatrix}.
\end{align*}
We test the performance of the preconditioners with two viscosity values, $\mu = 1$ and $10^{-4}$.


Table~\ref{Pdiag_2D} lists the number of MINRES iterations required to reach the specified tolerance for preconditioned systems
with  different choices of $\V{w}$ and block diagonal preconditioners.
As can be seen, the number for $\V{w}_1$, $\V{w}_2$, and $\V{w}_4$ remains small and consistent for meshes with various size
and $\mu$ changing from $1$ to $10^{-4}$. This is consistent with the theoretical result in Proposition~\ref{pro:MINRES_conv}
which states the convergence of MINRES is independent of $\mu$ and $h$ for finite $\gamma$.
On the other hand, while the required number of iterations for $\V{w}_3$ stays
relatively small (compared to the situation without preconditioning that requires more than 10000 iterations
for convergence for each case), it changes more significantly
for $\mu$ changing from $1$ to $10^{-4}$ and as the mesh is refined. This seems to reflect the theoretical result
of Proposition~\ref{pro:MINRES_conv-small} for small $\gamma$ which states that the convergence factor of MINRES iteration
is independent of $\mu$ and $h$ but its asymptotic error constant does depend on them although weakly.

Table~\ref{Ptri_2D} shows the number of GMRES iterations to reach the required tolerance.
The performance of GMRES is similar to that of MINRES and consistent with the theoretical results of Propositions~\ref{pro:GMRES_conv}
and \ref{pro:GMRES_conv-small}. Generally speaking, GMRES requires less iterations to reach convergence than MINRES.

\begin{table}[h]
    \centering
        \caption{The 2D Example: The number of MINRES iterations required to reach convergence for preconditioned systems with block diagonal preconditioners, $\mu = 1$ and $\mu = 10^{-4}$.}
    \begin{tabular}{|c|c|c|c|c|c|c|}
        \hline
         $\V{w}$ & \diagbox{$\mu$}{$N$} & 232 & 918 & 3680 & 14728 & 58608 \\ \hline \hline
        $\V{w}_1$ & $1$ & 43 & 47 & 49 & 49 & 47 \\  
        & $10^{-4}$ & 42  & 48 & 54 & 58 & 60 \\ \hline 
        \hline 
        $\V{w}_2$ & $1$ &  43 & 47 & 49 & 49 & 47 \\ 
        & $10^{-4}$ & 42  & 48 & 54 & 58 & 60 \\ \hline 
        \hline
        $\V{w}_3$ & $1$ &  62 & 68 & 49 & 49 & 47\\ 
        & $10^{-4}$ & 65 & 75 & 84& 92 & 98\\ \hline 
        \hline
        $\V{w}_4$ & $1$ &  44 & 47 & 49 & 49 & 47  \\ 
        & $10^{-4}$ & 42 & 48 & 54 & 58 & 60 \\ \hline 
    \end{tabular}
    \label{Pdiag_2D}
\end{table}

\begin{table}[h]
    \centering
        \caption{The 2D Example: The number of GMRES iterations required to reach convergence for preconditioned systems with block triangular preconditioners, $\mu = 1$ and $\mu = 10^{-4}$.}
    \begin{tabular}{|c|c|c|c|c|c|c|}
         \hline
         $\V{w}$ & \diagbox{$\mu$}{$N$} & 232 & 918 & 3680 & 14728 & 58608 \\ \hline
         \hline
        $\V{w}_1$ & $1$ & 21 & 23 & 24 & 25 & 25 \\  
        & $10^{-4}$ & 23 & 25 & 27 & 27 & 27 \\ \hline 
        \hline
        $\V{w}_2$ & $1$ & 21 & 23 & 24 & 25 & 25 \\ 
        & $10^{-4}$ & 23  & 25 & 27 & 27 & 27 \\ \hline 
        \hline
        $\V{w}_3$ & $1$ & 31 & 36 & 52 & 53 & 56\\ 
        & $10^{-4}$  & 33 & 38 & 55& 56 & 59\\ \hline 
        \hline
        $\V{w}_4$ & $1$ & 23 & 23 & 24 & 25 & 25  \\ 
        & $10^{-4}$ & 24 & 25 & 27 & 27 & 27 \\ \hline 
    \end{tabular}
    \label{Ptri_2D}
\end{table}



\subsection{The three-dimensional example}

This three-dimensional (3D) example is adopted from \textit{deal.II} \cite{dealii} \texttt{step-56}
where $\Omega = (0,1)^3$.
The exact solutions are defined as follows:
\begin{align*}
\V{u} = 
\begin{bmatrix}
    2 \sin(\pi x) \\
    -\pi y \cos(\pi x) \\
    -\pi z \cos(\pi x)
\end{bmatrix}, \quad
p = \sin(\pi x) \cos(\pi y) \sin(\pi z),
\end{align*}
and the right-hand side function is
\begin{align*}
\V{f} = 
\begin{bmatrix}
    2 \mu \pi^2 \sin(\pi x) + \pi \cos(\pi x) \cos(\pi y) \sin(\pi z) \\
    -\mu \pi^3 y \cos(\pi x) - \pi \sin(\pi y) \sin(\pi x) \sin(\pi z)\\
    -\mu \pi^3 z \cos(\pi x) + \pi \sin(\pi x) \cos(\pi y) \cos(\pi z)
\end{bmatrix}.
\end{align*}

The number of MINRES and GMRES iterations required to reach convergence for
the preconditioned systems for $\mu = 1$ and $10^{-4}$ is listed in Tables~\ref{Pdiag_3D} and
\ref{Ptri_3D}, respectively. Both MINRES and GMRES perform similarly as for the 2D example.
More specifically, the number for $\V{w}_1$, $\V{w}_2$, and $\V{w}_4$ (with finite $\gamma$) stays small and consistent for meshes of various size
and different values of $\mu$. On the other hand, the number for $\V{w}_3$ (with small $\gamma$) stays relatively small (compared to those without
preconditioning) but changes more significantly for different values of $\mu$ and $h$. These results are consistent with the theoretical analysis
given in Propositions~\ref{pro:MINRES_conv}, \ref{pro:GMRES_conv}, \ref{pro:MINRES_conv-small}, and \ref{pro:GMRES_conv-small}.

\begin{table}[h]
    \centering
        \caption{The 3D Example: The number of MINRES iterations required to reach convergence for preconditioned systems with block diagonal preconditioners, $\mu = 1$ and $\mu = 10^{-4}$.}
    \begin{tabular}{|c|c|c|c|c|c|c|}
        \hline
        $\V{w}$ & \diagbox{$\mu$}{$N$} & 4046 & 7915 & 32724  & 112078 & 266555 \\ \hline
         \hline
        $\V{w}_1$ & $1$ & 59&59  &63  & 67 & 67 \\ 
        & $10^{-4}$ & 62 &62  & 70 &  76 &78  \\ \hline 
        \hline
        $\V{w}_2$ & $1$ & 59 &59 & 63 & 67 & 67\\ 
        & $10^{-4}$ &62  & 62& 70 & 76 &78  \\ \hline 
        \hline
        $\V{w}_3$ & $1$ & 59&59  &63  & 67 & 67\\ 
        & $10^{-4}$  &110 & 114& 126& 76& 80\\ \hline 
        \hline
         $\V{w}_4$ & $1$ & 59&59  &63  & 67 & 67 \\ 
         & $10^{-4}$ & 60 &62  & 70 &  76 &78  \\
         \hline
    \end{tabular}
    \label{Pdiag_3D}
\end{table}

\begin{table}[h]
    \centering
        \caption{The 3D Example: The number of GMRES iterations required to reach convergence for preconditioned systems with block triangular preconditioners, $\mu = 1$ and $\mu = 10^{-4}$.}
    \begin{tabular}{|c|c|c|c|c|c|c|}
        \hline
         $\V{w}$ & \diagbox{$\mu$}{$N$} & 4046 & 7915 & 32724  & 112078 & 266555 \\ \hline
         \hline
        $\V{w}_1$ & 1 & 30 & 30 & 31  & 32 & 33 \\ 
        & $10^{-4}$ & 34  & 35 & 37 & 38 & 38\\ \hline 
        \hline
        $\V{w}_2$ & $1$ & 30 &30 & 31 & 32 & 33\\ 
        & $10^{-4}$ &34  & 35& 37 & 38 &38   \\ \hline 
        \hline
        $\V{w}_3$ & $1$ &53 &59 &60 & 63 & 34\\ 
        & $10^{-4}$ &55 & 59& 65& 89& 271 \\ \hline 
        \hline
         $\V{w}_4$ & $1$ &30  & 30 & 31 &  32 & 33 \\
         & $10^{-4}$ & 35& 35 & 37 & 38 & 38 \\ 
         \hline
    \end{tabular}
    \label{Ptri_3D}
\end{table}

\section{Conclusions}
\label{SEC:conclusions}

In the previous sections we have studied the convergence of MINRES and GMRES for the iterative solution of
the lowest-order weak Galerkin finite element approximation for Stokes flow with general regularization.
The original saddle point system \eqref{scheme_matrix} of the Stokes flow is singular,
as shown in Lemma~\ref{lem:B0-1}.
Moreover, when a nonhomogeneous Dirichlet boundary condition is used,
the discrete system is inconsistent in general due to the error in the numerical approximation of the boundary datum.
This inconsistency and singularity cause difficulties in solving the system.
To address these issues, we have proposed a general regularization strategy with which a term $-\frac{\rho}{\mu} \V{w}\V{w}^T$
is added to the zero $(2,2)$-block, where $\V{w}$ is a unit vector satisfying $\V{w}^T \V{1} \neq 0 $ (see \eqref{w_def}),
$\V{1}$ is a unit basis vector of the null space of $(B^{\circ})^T$, and $\rho > 0$ is a regularization constant.
The regularization can be viewed as an attempt to enforce $\V{w}^T \V{p}_h = 0$, i.e., projecting the pressure variable
into the orthogonal complement of $\text{span}(\V{w})$.
Proposition~\ref{thm:reg_scheme_err} states that the regularized system \eqref{scheme_reg} can achieve the optimal-order
convergence provided that the nonhomogeneous Dirichlet boundary datum is approximated numerically with sufficient accuracy.

We have considered block diagonal and triangular Schur complement preconditioning for the efficient iterative solution of the regularized system \eqref{scheme_reg}.
The situation with finite $\rho$ and $\gamma$ was studied in Section~\ref{sec:finite-gamma}.
With block diagonal Schur complement preconditioning, bounds for the eigenvalues of the preconditioned system
and for the residual of MINRES are given in Lemma~(\ref{lem:eigen_bound_diag}) and Proposition~\ref{pro:MINRES_conv}.
These bounds demonstrate that the convergence of MINRES is essentially independent of $h$ and $\mu$.
Additionally, the convergence of GMRES for the regularized system with block triangular Schur complement preconditioning
has been studied and shown to be essentially independent of $h$ and $\mu$ (cf. Proposition~\ref{pro:GMRES_conv}).

The situation with small $\gamma$ was studied in Section~\ref{sec:small-gamma}. In this situation, $\rho$ is taken to be small
and the regularization term is considered as a perturbation to the original singular system.
Convergence estimates of the residual for MINRES and GMRES are given in Propositions~\ref{pro:MINRES_conv-small} and \ref{pro:GMRES_conv-small}, respectively.
These bounds indicate that the convergence factor is essentially independent of $h$ and $\mu$
while the number of iterations required to reach convergence depends on $\log(\gamma^2 N \rho)$.

The numerical results in two and three dimensions were presented in Section~\ref{SEC:numerical}, where both the finite and small $\gamma$ were tested.
The effectiveness of block Schur complement preconditioning for the regularized system was demonstrated.

It is worth noting that for the efficient solution of the regularized system \eqref{scheme_reg}, a choice of $\V{w}$ with finite
$\gamma = \V{w}^T \V{1}$ should be used. With such a choice, the theoretical and numerical analysis presented in this work
shows that the convergence of MINRES and GMRES with block Schur complement preconditioning is independent of $\mu$ and $h$.

\section*{Acknowledgments}

W.~Huang was supported in part by the Simons Foundation grant MPS-TSM-00002397.


\appendix

\section{Appendix: Proof of Lemma~\ref{eigen_bound2}}
\label{sec:eigen_bound2-proof}

\begin{proof}
Lemma~\ref{lem:S-22} implies that the upper bound of the eigenvalues of $\hat{S}^{-1}S$ is $d$.
    
For the lower bound, denote the eigenpairs of $B^{\circ}(A)^{-1} (B^{\circ})^T$
by $(\V{u}_i, \lambda_i)$ ($i = 1,2,\dots,N$ and $\lambda_1\le \lambda_2 \le \cdots \le \lambda_N$), i.e.,
    \begin{align*}
        B^{\circ}(A)^{-1} (B^{\circ})^T \V{u}_i = \lambda_i \V{u}_i,\quad i = 1,2,\dots, N.
    \end{align*}
From Lemma~\ref{lem:B0-1} we have $\lambda_1 = 0$ and  $\V{u}_1 = \V{1}$ (cf. \eqref{vector-1}).
For other eigenvalues, recall that the non-zero eigenvalues of $(M_p^{\circ})^{-1}B^{\circ}(A)^{-1} (B^{\circ})^T$ are located in
$[\beta^2,d]$ \cite{HuangWang_CiCP_2025}, where $\beta$ is the inf-sup constant.
Thus, the non-zero eigenvalues of $B^{\circ}(A)^{-1} (B^{\circ})^T$ are bounded below by
    \begin{align}
        \lambda_i \geq \beta^2 \lambda_{\min}(M_p^{\circ}), \quad i = 2, \dots, N.
        \label{eigen_bound3}
    \end{align}
Moreover, the unit vector $\V{w}$ can be expressed as
\begin{align}
    \V{w} = \gamma \V{1} +   \sum_{i = 2}^N b_i \V{u}_i,
    \quad \text{with} \quad \gamma^2 + \sum_{i = 2}^N b_i^2 = 1, \quad \gamma = \V{1}^T \V{w} \neq 0 .
\label{b1-1}
\end{align}

Consider an arbitrary nonzero vector $\V{v}$
\[
\V{v} = \sum_{i = 1}^N a_i \V{u}_i .
\]
Using \eqref{eigen_bound3} and the fact that the mass matrix $M_p^{\circ}$ is diagonal,
$\lambda_{\min}(M_p^{\circ}) = \min_{K} |K|$, and $\lambda_{\max}(M_p^{\circ}) = \max_{K}|K|$, we have
 \begin{align}
     \frac{\V{v}^T S \V{v}}{\V{v}^T \hat{S} \V{v}}
     & = \frac{\V{v}^T ( \rho \V{w}\V{w}^T + B^{\circ}(A)^{-1} (B^{\circ})^T) \V{v}}{\V{v}^T\Big( \rho \V{w}\V{w}^T + M_p^{\circ}\Big) \V{v}}
     \displaystyle \notag
     \\ &\geq 
      \frac{\displaystyle \rho \Big(a_1 \gamma + \sum_{i = 2}^N a_i b_i\Big)^2 +  \sum_{i = 2}^N \lambda_i a_i^2}{\displaystyle\rho \Big(a_1 \gamma + \sum_{i = 2}^N a_i b_i\Big)^2 + \lambda_{\max}(M_p^{\circ}) \sum_{i=1}^N  a_i^2} \notag
     \\
     & \geq 
      \frac{\displaystyle \rho \Big(a_1 \gamma + \sum_{i = 2}^N a_i b_i\Big)^2 + \beta^2 \lambda_{\min}(M_p^{\circ}) \sum_{i = 2}^N  a_i^2}{\displaystyle\rho \Big(a_1 \gamma + \sum_{i = 2}^N a_i b_i\Big)^2+ \lambda_{\max}(M_p^{\circ}) \sum_{i=1}^N  a_i^2}.
     \label{eigen_bound1}
    \end{align}
We need to simplify the right-hand side further.
Denote $\V{a} = \begin{bmatrix}
        a_1, \dots,a_N
    \end{bmatrix}^T$ and $\V{b} = \begin{bmatrix}
        \gamma,b_2,\dots,b_n
    \end{bmatrix}^T.$
Express $\V{a}$ in terms of $\V{b}$ into
\[
\V{a} = \alpha \V{b} + \V{c},
\] 
where $\alpha = \V{b}^T \V{a}$ and $\V{b}^T \V{c} = 0$, i.e.,
\begin{align}
    \alpha = \V{b}^T \V{a} = a_1 \gamma +  \sum_{i = 2}^N a_i b_i, \quad \gamma c_1 = - \sum_{i = 2}^N b_i c_i. 
\label{abc_term}
\end{align}
From this, (\ref{b1-1}), and the Cauchy-Schwarz inequality, we have
    \begin{align}
        \displaystyle \sum_{i = 1}^N a_i^2 
       & = (\alpha \V{b} + \V{c})^T (\alpha \V{b} + \V{c})
       = \alpha^2 + \sum_{i=1}^N c_i^2 = \alpha^2 + \sum_{i=2}^N c_i^2 + c_1^2 \notag
        \\
        & = \alpha^2 + \sum_{i=2}^N c_i^2 + \frac{1}{\gamma^2}{\Big(\sum_{i=2}^N b_ic_i\Big)^2} \notag
        \\
        & \leq \alpha^2 + \sum_{i=2}^N c_i^2 + \frac{1}{\gamma^2} \sum_{i=2}^N b_i^2 \sum_{i=2}^N c_i^2 \notag
        \\
        & = \alpha^2 + \sum_{i=2}^N c_i^2 + \frac{1}{\gamma^2} (1-\gamma^2) \sum_{i=2}^N c_i^2 \notag
        \\
        & = \alpha^2  + \frac{1}{\gamma^2}  \sum_{i=2}^N c_i^2.
        \label{ai_1}
    \end{align}
    Similarly, we have
    \begin{align}
        \sum_{i = 2}^N a_i^2 & = \sum_{i = 1}^N a_i^2 - a_1^2
        = \alpha^2 + \sum_{i=1}^N c_i^2  - (\alpha \gamma + c_1)^2 \notag
        \\ & = \alpha^2(1-\gamma^2) + \sum_{i=2}^N c_i^2 + 2\alpha \sum_{i=2}^N b_i c_i \notag
        \\ & \geq \alpha^2(1-\gamma^2) + \sum_{i=2}^N c_i^2 - 2 |\alpha| \Big(\sum_{i=2}^N b_i^2\Big)^{\frac{1}{2}} \Big(\sum_{i=2}^N c_i^2\Big)^{\frac{1}{2}} \notag
        \\ & = \alpha^2(1-\gamma^2) + \sum_{i=2}^N c_i^2 - 2 |\alpha|  \Big(1-\gamma^2\Big)^{\frac{1}{2}} \Big(\sum_{i=2}^N c_i^2\Big)^{\frac{1}{2}} \notag
        \\
        & = \Bigg( \Big(\sum_{i=2}^N c_i^2\Big)^{\frac{1}{2}} - |\alpha| \Big({1-\gamma^2}\Big)^{\frac{1}{2}}\Bigg)^2 .
        \label{ai_2}
    \end{align}
    Substituting \eqref{ai_1} and \eqref{ai_2} into \eqref{eigen_bound1}, we get
    \begin{align}
        \frac{\V{v}^T S \V{v}}{\V{v}^T \hat{S} \V{v}}
     & \geq \frac{\displaystyle \rho \alpha^2 + \beta^2 \lambda_{\min}(M_p^{\circ}) \Big( (\sum_{i=2}^N c_i^2)^{\frac{1}{2}} - |\alpha| (1-\gamma^2)^{\frac{1}{2}} \Big)^2}{\displaystyle \rho \alpha^2 + \frac{\displaystyle \lambda_{\max}(M_p^{\circ})}{\gamma^2} \Big( \alpha^2 \gamma^2  + \sum_{i=2}^N c_i^2\Big)}.
     \label{min_1}
    \end{align}
To find the minimum of the right-hand side of \eqref{min_1}, we define
    \[
    x = \frac{ \displaystyle \Big(\sum_{i=2}^N c_i^2\Big)^{\frac{1}{2}}}{|\alpha|}.
    \] 
Notice that $x$ is non-negative.
    Then, we can rewrite \eqref{min_1} as
    \begin{align}
        \frac{\V{v}^T S \V{v}}{\V{v}^T \hat{S} \V{v}}
     & \geq \frac{\rho + \beta^2 \lambda_{\min}(M_p^{\circ}) \Big( x - (1-\gamma^2)^{\frac{1}{2}}\Big)^{2}}{\rho + \displaystyle\frac{\lambda_{\max}(M_p^{\circ})}{\gamma^2} (\gamma^2 + x^2)}.
     \label{min_2}
    \end{align}
    For simplicity of notation, we denote 
    \begin{align}
          & \delta = (1-\gamma^2)^{\frac{1}{2}} \geq 0, \quad \sigma = \frac{\lambda_{\max}(M_p^{\circ})}{\gamma^2} > 0, \quad \eta = \beta^2 \lambda_{\min}(M_p^{\circ}) > 0,
    \label{min_3}  
    \\
    & f(x) = \frac{\rho + \beta^2 \lambda_{\min}(M_p^{\circ}) \Big( x - (1-\gamma^2)^{\frac{1}{2}}\Big)^{2}}{\rho + \displaystyle\frac{\lambda_{\max}(M_p^{\circ})}{\gamma^2} (x^2 + \gamma^2)}
     = \frac{\rho + \eta (x-\delta)^2}{\rho + \sigma (x^2+\gamma^2)}.
     \notag
    \end{align}
We would like to find a lower bound of $f(x)$ on $[0,+\infty)$. To this end, we find the first derivative of this function as
    \begin{align*}
        f'(x) & = \frac{2 \eta (x-\delta) \Big(\rho+\sigma (x^2 + \gamma^2)\Big) - 2\sigma x \Big(\rho+\eta (x-\delta)^2\Big)}{\Big(\rho + \sigma (x^2+\gamma^2)\Big)^2}
        = \frac{2 \; g(x)}{\Big(\rho + \sigma (x^2+\gamma^2)\Big)^2} ,
    \end{align*}
    where
    \begin{align*}
    g(x) = \sigma \delta \eta x^2 - \Big(\rho (\sigma - \eta) + \eta \sigma (\delta^2 - \gamma^2)\Big)x - \delta \eta  (\rho + \sigma \gamma^2).
    \end{align*}

We first consider the case $\delta = 0$ (i.e., $\gamma^2 = 1$). For this case, 
    \[
    g(x) = - \Big(\rho (\sigma - \eta) - \eta \sigma \Big)x.
    \]
Notice that $\sigma - \eta >0$, $\sigma = \mathcal{O}(h^d)$, $\eta = \mathcal{O}(h^d)$, and $\sigma \eta = \mathcal{O}(h^{2d})$ as $h \to 0$.
It follows that $g(x) \le 0$ and $f(x)$ decreases on $[0,+\infty)$. Thus,
\begin{align}
f(x) \ge f(+\infty) = \frac{\eta}{\sigma} = \frac{\beta^2 \lambda_{\min} (M_p^{\circ})}{\lambda_{\max} (M_p^{\circ})} .
\label{fmin-1}
\end{align}

Next, we consider the case $\delta > 0$.
Setting $g(x) = 0$ yields two different roots, 
    \begin{align*}
    x_1 = \frac{ \Big( \rho (\sigma - \eta) + \sigma \eta (\delta^2 - \gamma^2) \Big) - \sqrt{ \Big( \rho (\sigma - \eta) + \sigma \eta (\delta^2 - \gamma^2) \Big)^2 + 4 \sigma \delta^2 \eta^2 (\rho + \sigma \gamma^2)} }{2 \sigma \delta \eta},
    \\
    x_2 = \frac{ \Big( \rho (\sigma - \eta) + \sigma \eta (\delta^2 - \gamma^2) \Big) + \sqrt{ \Big( \rho (\sigma - \eta) + \sigma \eta (\delta^2 - \gamma^2) \Big)^2 + 4 \sigma \delta^2 \eta^2 (\rho + \sigma \gamma^2)} }{2 \sigma \delta \eta}.
    \end{align*}
Since $x_1 < 0$ and $x_2 > 0$, we only need to consider $x_2$.
After some algebraic manipulation, we obtain 
    \[
    x_2 \approx \frac{\rho (\sigma - \eta)}{\sigma \delta \eta} = {\mathcal{O}(h^{-d})} .
    \]
Since $g(0) < 0$, $g(x_2) = 0$, and $g(+\infty) > 0$, we know that $f(x)$ is decreasing on $(0,x_2)$ and increasing on $(x_2, +\infty)$.
Thus,
    \begin{align*}
        f(x) \ge f(x_2) = \frac{\rho + \eta (x_2-\delta)^2}{\rho + \sigma (x_2^2+\gamma^2)}
        = \frac{\eta}{\sigma} + \mathcal{O}\big (\frac{1}{x_2}\big ) 
        = \beta^2  \frac{\lambda_{\min} (M_p^{\circ})}{\lambda_{\max} (M_p^{\circ})} \gamma^2 + \mathcal{O}(h^d) .
    \end{align*}
Combining this and (\ref{fmin-1}) we obtain (\ref{eigen_bound2-eq}).
\end{proof}



\end{document}